\documentclass[12pt]{article}
\usepackage{amsfonts}
\usepackage{amsmath,amsthm}
\usepackage{cases}

\newtheorem{lemma}{Lemma}[section]

\newtheorem{theorem}{Theorem}[section]
\newtheorem{proposition}[theorem]{Proposition}
\theoremstyle{definition}
\newtheorem{definition}{Definition}[section]
\newtheorem{corollary}[theorem]{Corollary}
\newtheorem{remark}{Remark}

\newcommand{\comment}[1]{}
\newcommand{\R}{\mathbb{R}}

\newcommand{\Gammatt}{\widetilde{\Gamma}_2}
\newcommand{\asum}[1]{\widetilde{\sum_{#1}}}
\numberwithin{equation}{section}
\begin{document}
\title{Volume doubling, Poincar\'{e} inequality and Gaussian heat kernel estimate for non-negatively curved graphs}
\author{Paul Horn, Yong Lin\footnotemark[1], Shuang Liu, Shing-Tung Yau\footnotemark[2]}
\date{}
\maketitle

\renewcommand{\thefootnote}{\fnsymbol{footnote}}
\footnotetext[1]{Supported by the National Natural Science Foundation of China(Grant No$11271011$).}
\footnotetext[2]{Supported by the NSF DMS-0804454.}

\begin{center}
\textbf{Abstract}
\end{center}

By studying the heat semigroup, we prove Li-Yau type estimates for  bounded and positive solutions of the heat equation on graphs, under the assumption of the curvature-dimension inequality $CDE'(n,0)$, which can be consider as a notion of curvature for graphs. Furthermore, we derive that if a graph has non-negative curvature then it has the volume doubling property, from this we can prove the Gaussian estimate for heat kernel, and then Poincar\'{e} inequality and Harnack inequality. As a consequence, we obtain that the dimension of space of harmonic functions on graphs with polynomial growth is finite, which original is a conjecture of Yau on Riemannian manifold proved by Colding and Minicozzi. Under the assumption of positive curvature on graphs, we derive the Bonnet-Myers type theorem that the diameter of graphs is finite and bounded above in terms of the positive curvature by proving some Log Sobolev inequalities.

\section{Introduction}

The Li-Yau inequality is a very powerful tool for studying positive solutions to the heat equation on manifolds. In its simplest case, it states that a positive solution $u$ (that is a positive $u$ satisfying $\partial_t  u = \Delta u$) on a compact $n$-dimensional manifold with non-negative curvature satisfies
\begin{equation}\label{eq:ly}
\frac{|\nabla u|^2}{u^2}-\frac{\partial_t u}{u}\leq \frac{n}{2t}.
\end{equation}
Beyond its utility in the study of Riemannian manifolds, variants of the Li-Yau inequality have proven to be an important tool in non-Riemannian settings as well.  Recently, in \cite{BHLLMY13}, the authors have proved a discrete version of Li-Yau inequality on graphs.  The discrete setting provided myriad challenges, with many of these stemming from the lack of a chain rule for the Laplacian in a graph setting.  Overcoming this, involved introducing a new notion of curvature for graphs, and exploited crucially the fact that a chain rule formula for the Laplacian does hold in a few isolated case, along with a  discrete version of maximum principle. Indeed, while there are two main methods known to prove the gradient estimate (\ref{eq:ly}) -- one being the maximum principle (as in \cite{LY86} on manifolds and \cite{LY10} on graphs), and the other being semigroup methods (\cite{BL06} on manifolds) -- the standard application of both techniques relies heavily on the chain rule and the continuous nature of the underlying space.

The Li-Yau inequality has many applications in Riemannian geometry, but among the most important of these is establishing Harnack inequalities. Indeed, inequality~\eqref{eq:ly} can be integrated over space-time in order to derive Harnack inequalities, yielding an inequality of the form:
\begin{equation} \label{eq:harnack}
u(x,s) \leq C(x,y,s,t) u(y,t),
\end{equation}
where $C(x,y,s,t)$ depends only on the distance of $(x,s)$ and $(y,t)$ in space-time.
The Li-Yau inequality, and more generally of parabolic Harnack inequalities like \eqref{eq:harnack} can also be used to derive further heat kernel estimates.  In this direction, one of the most important estimates to achieve are the following Gaussian type bounds:
\begin{equation} \label{eq:gaussian}
\frac{c_{l}m(y)}{V(x,\sqrt{t})}e^{-C_{l}\frac{d(x,y)^{2}}{t}}\leq p(t,x,y)\leq \frac{C_{r}m(y)}{V(x,\sqrt{t})}e^{-c_{r}\frac{d(x,y)^{2}}{t}},
\end{equation}
where $p(t,x,y)$ is a fundamental solution of the heat equation (heat kernel).  The Li-Yau inequality can be used to prove exactly such bounds for the heat kernel on non-negatively curved manifolds.  Thus, via the Li-Yau inequality it can be shown that non-negatively curved manifolds satisfy a strong form of the Harnack inequality~\eqref{eq:harnack}, along with a Gaussian estimate~\eqref{eq:gaussian}.  It also is known, by combining the Bishop-Gromov comparison theorem \cite{Bi63} and the work of Buser \cite{Bu82} that non-negatively curved manifolds also satisfy the volume growth condition known as volume doubling and the Poincar\'e inequality (see also the paper of Grigor'yan, \cite{G92}).

In manifold setting, Grigor'yan~\cite{G92} and Saloff-Coste~\cite{SC95} independently gave a complete characterization of manifolds satisfying \eqref{eq:harnack}. They showed that satisfying a volume doubling property along with Poincar\'e inequalities is actually {\it equivalent} to satisfying the Harnack inequality~\eqref{eq:harnack}, and is also equivalent to satisfying the Gaussian estimate~\eqref{eq:gaussian}.  Thus, in the manifold setting the three conditions discussed above that are implied by non-negative curvature are actually all equivalent.  Curvature still plays an important role however, as a local property that certifies that a manifold satisfies the three (equivalent) global properties.

In the case of graphs, Delmotte~\cite{D99} proved a characterization analogous to that of Saloff-Coste for both continuous and discrete time.  Until now, however, no known notion of curvature on graphs has been sufficient to imply that a graph satisfies these three conditions.  This is not to say that the question of whether some sort of curvature lower bound implies strong geometric properties in a non-Riemannian setting.  On metric measure spaces, for instance, under some curvature lower bound assumptions, Sturm \cite{S06}, Rajala \cite{R12}, Erbar, Kuwada and  Sturm \cite{EKS13} and Jiang, Li and Zhang \cite{JLZ14} studied the volume doubling property along with Poincar\'e inequalities and Gaussian heat kernel estimates.

Despite the successes of \cite{BHLLMY13} in establishing a discrete analogue of the Li-Yau inequality, their ultimate result also had some limitations.  Most notably, the results of \cite{BHLLMY13} were unable insufficient to derive the equivalent conditions of volume doubling and Poincar\'e inequalities, along with Gaussian heat kernel bounds, and the strongest form of a Harnack inequality.  This failure arose from the generalization of \eqref{eq:ly} achieved when considering only a positive solution inside a ball of radius $R$: in the classical case an extra term of the form `$\frac{1}{R^2}$' occurred, but in the graph case in general the authors were only able to prove a result with an extra term of the factor `$\frac{1}{R}$'.  This difference resulted in only being able to establish weaker bounds on the heat kernel, and polynomial volume growth as opposed to the stronger condition of volume doubling.  Ultimately one of the reasons for these weaker implications was the methods used:~\cite{BHLLMY13} used maximum principle arguments, and ultimately ran into problems when cutoff functions were needed.

In this paper, we develop a way to apply semigroup techniques in the discrete setting in order to study the heat kernel of graphs with non-negative Ricci curvature.  From here, we obtain a family of global gradient estimates for bounded and positive solutions to the heat equation on an infinite graph.  The curvature notion used, as in \cite{BHLLMY13}, is a modification of the so-called curvature dimension inequality.  Satisfying a curvature dimension inequality has proven to be an important generalization of having a Ricci curvature lower bound in the non-Riemannian setting (see, eg. \cite{BE83, BL06}).  The utility of satisfying the standard such inequality is much lower when the Laplace operator does not satisfy the chain rule, such as in the graph case.  This led to the modification used in this paper (and in \cite{BHLLMY13}) the so-called {\it exponential curvature dimension inequalities}.  A more detailed description of the curvature notion used in this paper, and the motivation behind it, is given in Section \ref{sectionCDE}. It is, however, important to note that in the Riemannian case (and more generally when the Laplacian generates a diffusive semigroup) the classical curvature dimension inequality, and the exponential curvature dimension inequalities are equivalent.

From our new methods,  we show that non-negatively curved graphs (in the sense of the exponential curvature dimension inequalities) satisfy volume doubling.  This improves the results of~\cite{BHLLMY13}, in which only polynomial volume growth is derived.   This improvement is the key point in proving the discrete-time Gaussian lower and upper estimates of heat kernel, and from this the Poincar\'{e} inequality and Harnack inequality on graphs.  As an important technical point, we do not simply establish volume doubling and the Poincar\'e inequality and then apply the results of Delmotte~\cite{D99} to establish the other (equivalent conditions).  Instead, after proving volume doubling we attack the Gaussian bounds directly -- using volume doubling along with additional information from our methods to establish the Gaussian bounds.  Once Gaussian bounds are established, however, we can use the results of Delmotte to `complete the circle,' and establish the remaining desired properties.  We emphasize that although a number of notions of curvature for graphs have been introduced (see, eg,\cite{LY10, BHLLMY13}) no previous notion has been shown to imply these conditions -- and in fact, \cite{BHLLMY13} was the first paper to show that a non-negative curvature condition for graphs implied polynomial volume growth.

We further derive continuous-time Gaussian lower estimate of heat kernel.  It proves to be impossible, however, to prove continuous-time Gaussian {\it upper} bounds on the heat kernel, however, as from the paper of Davies \cite{DB93} and Pang \cite{P93} the continuous-time Gaussian upper estimate is not true on graphs.

While we prove this for any non-negatively curved graph, it is important to note that Gaussian estimates for the heat kernel for Cayley graphs of a finite generated group of polynomial growth were proved by Hebisch and
Saloff-Coste in \cite{HS93}. For non-uniform transition case, Strook and Zheng proved related Gaussian estimates on lattices in \cite{SZ97}.

Establishing that a graph satisfies both volume doubling and the Poincar\'{e} inequality has important consequences.  For example, under these assumptions on graphs, Delmotte in \cite{D97} proved that the dimension of the space of Harmonic function on graphs with polynomial growth is finite.  This extends the similar result on Riemannian manifolds by Colding and Minicozzi in \cite{CM96}, see also Li in \cite{Li97}.  The original problem came from a conjecture of Yau (\cite{Yau86}) which stated that these space should have finite dimension in Riemannian manifolds with non-negative Ricci curvature.  Thus, our result answers the analogue conjecture of Yau for graphs in the affirmative.

Finally, under the assumption of a graph being positively curved (again, with respect to the exponential curvature dimension inequality), we derive a Bonnet-Myers type theorem that the diameter of graphs in terms of the canonical distance is finite.  We accomplish this by proving some logarithmic Sobolev inequalities.  Here we establish that certain diameter bounds of Bakry still hold, even though the Laplacian on graphs does not satisfy the diffusion property that Bakry used.  Under the same assumption, we also can prove that the diameter of graphs in terms of graph distance is finite by proving the finiteness of measure, plus the doubling property of volume.

The paper is organized as follows: We introduce our notation and formally state our main results in Section \ref{sec:statement}.  In Section \ref{sec:LY}, we prove our main variational inequality.  This inequality leads to a different proof of the Li-Yau gradient estimates on graphs from the one given in \cite{BHLLMY13}.  From this main inequality we establish an additional exponential integrability result, and ultimately, volume doubling in Section \ref{sec:VG}. From volume doubling, we can prove the Gaussian heat kernel estimate, parabolic
Harnack inequality and Poincar\'{e} inequality in Section \ref{sec:GE}.  Finally, in Section \ref{sec:DB}, we prove a Bonnet-Myers type theorem on graphs.

{\bf Acknowledgments} We thank Bobo Hua, Matthias Keller and Gabor Lippner for useful discussion. We also thank Daniel Lenz for many nice comments on the
paper. Part of the work of this paper was done when P. Horn and Y. Lin visited S.-T. Yau in The National Center for Theoretical Sciences in Taiwan University in May 2014 and when P. Horn visited Y. Lin in Renmin University of China in June 2014. We acknowledge the support from NCTS and Renmin University.

\section{Preliminaries and Statement of main results} \label{sec:statement}
In this section we develop the preliminaries needed to state our main results.  Through the paper, we let $G=(V,E)$ be a finite or infinite connected graph.  We allow the edges on the graph to be weighted.  Weights are given by a function $\omega: V\times V\rightarrow [0,\infty)$, the edge $xy$ from $x$ to $y$ has weight $\omega_{xy}>0$.  In this paper, we assume this weight function is symmetric (that is, $\omega_{xy}=\omega_{yx}$).  Furthermore, we assume that
$$\omega_{\min}=\inf_{e \in E: \omega_e > 0}\omega_e>0.$$
We furthermore allow loops, so it is permissible for $x \sim x$ (and hence $\omega_{xx} > 0$.)  Finally, we restrict our interest to locally finite.  That is, we assume that
$$m(x):=\sum_{y\sim x}\omega_{xy}<\infty, \quad \forall x\in V.$$

For our work, especially in the context of deriving Gaussian heat kernel bounds, one additional technical assumption is needed.  This is essentially needed to compare the continuous time and discrete time heat kernels.  In order for this to go smoothly two things need to happen: no edge can be too `small' (this is essentially the content of our assumption $\omega_{min} > 0$), and also at each vertex there must be a loop.  That is, we must assume $x \sim x$ -- this prevents `parity problems' of bipartiteness that would make the continuous and discrete time kernels incomparable.  This condition is neatly captured in the following $\Delta(\alpha)$ used by Delmotte in \cite{D99}, but has also been used previously by other authors.

\begin{definition}
Let $\alpha>0$.  $G$ satisfies $\Delta(\alpha)$ if,
\begin{enumerate}
\item[1)] $x \sim x$ for every $x \in V$, and
\item[2)] If $x,y \in V$, and $x\sim y$,
\[\omega_{xy}\geq \alpha m(x).\]
\end{enumerate}
\end{definition}
As a remark, if a loop is on every edge and $\sup_x m(x) < \infty$, then the condition $\omega_{\min} > 0$ is sufficient to certify that a graph satisfies $\Delta(\omega_{\min}/\sup_x m(x)).$  In general, this is a rather mild condition. It is easy to check, for instance, that adding loops does not decrease the curvature for our curvature condition (see Section \ref{sectionCDE} below) nor change many the geometric quantities we seek to understand (eg. volume growth, and diameter).  Thus even graphs without loops may safely be altered to satisfy this condition.

\subsection{Laplace Operators on Graphs}

Let $\mu:V \rightarrow \mathbb{R^{+}}$ be a positive measure on the vertices of the $G$. We denote by $V^{\mathbb{R}}$ the space of real functions on $V$. and we denote by $\ell^{p}(V,\mu)=\{f \in V^{\mathbb{R}}:\sum_{x\in V}\mu(x)|f(x)|^{p}<\infty\}$, for any $1\leq p< \infty$,
the space of $\ell^{p}$ integrable functions on $V$ with respect to the measure $\mu$.  For $p=\infty$, let $\ell^{\infty}(V,\mu)=\{f \in V^{\mathbb{R}}:\sup_{x\in V}|f(x)|<\infty\}$ be the set of bounded functions. For any $f,g \in \ell^{2}(V,\mu)$, we let $\langle f,g\rangle=\sum_{x\in V}\mu(x)f(x)g(x)$ denote the standard inner product.  This makes $\ell^{2}(V,\mu)$ a Hilbert space.  As is usual, we can define the $\ell^p$ norm of $f\in \ell^{p}(V,\mu),1\leq p\leq \infty$:
\[\|f\|_p=\left(\sum_{x\in V}\mu(x)|f(x)|^p\right)^{\frac{1}{p}}, 1\leq p<\infty ~\mbox{and}~ \|f\|_{\infty}=\sup_{x\in V}|f(x)|.\]

We define the $\mu-$Laplacian $\Delta:V^{\mathbb{R}}\rightarrow V^{\mathbb{R}}$ on $G$ by, for any $x\in V$,
$$\Delta f(x)=\frac{1}{\mu(x)}\sum_{y\sim x} \omega_{xy}(f(y)-f(x)).$$
Similar summations occur frequently, so we introduce the following shorthand notation for such an ``averaged sum.''
$$\widetilde{\sum_{y\sim x}}h(y)=\frac{1}{\mu(x)}\sum_{y\sim x} \omega_{xy}h(y) \quad \forall x\in V.$$

Under our locally finite assumption, it is clear that for any bounded $f$, $\Delta f$ is likewise bounded.  We treat the case of $\mu$ Laplacians quite generally, but the two most natural choices are the case where $\mu(x)=m(x)$ for all $x\in V$, which is the normalized graph Laplacian, and the case $\mu$ $\equiv 1$ which is the standard graph Laplacian.  Furthermore, in this paper we assume
$$D_{\mu}:=\max_{x \in V}\frac{m(x)}{\mu(x)}<\infty.$$
It is easy to check that $D_{\mu}<\infty$ is equivalent to the Laplace operator $\Delta$ being bounded on $\ell^{2}(V,\mu)$ (see also \cite{HKLW12}).
The graph is endowed with its natural graph metric $d(x,y)$, i.e. the smallest number of edges of a path between two vertices $x$ and $y$. We define balls $B(x,r) = \{y\in V:d(x,y)\leq r\}$, and the volume
of a subset $A$ of $V$, $V(A) =\sum_{x\in A}\mu(x)$. We will write $V(x,r)$ for $V(B(x,r))$.



\subsection{Curvature Dimension Inequalities} \label{sectionCDE}

In order to study curvature of non-Riemannian spaces, it is important to have a good definition that allows one to capture the important consequences.  One way to do this is through the so-called {\it curvature-dimension inequality} or CD-inequality.  An immediate consequence of the well-known Bochner identity is that on any $n$-dimensional manifold with curvature bounded below by $K$, any smooth $f:M \to \mathbb{R}$ satisfies:

\begin{equation} \label{eq:cd-ine}
\frac{1}{2}\Delta|\nabla f|^2\geq\langle\nabla f,\nabla\Delta f\rangle+\frac{1}{n}(\Delta f)^2+K|\nabla f|^2.
\end{equation}
It was an important insight by Bakry and Emery~\cite{BE83} that one can use \ref{eq:cd-ine} as a substitute for a lower Ricci curvature bound on spaces where a direct generalization of Ricci curvature is not available. Since all known proofs of the Li-Yau gradient estimate exploit non-negative curvature condition through the $CD$-inequality, Bakry and Ledoux~\cite{BL06} succeeded to use it to generalize~\eqref{eq:ly} to Markov operators on general measure spaces when the operator satisfies a chain rule type formula.

To formally introduce this notion for graphs, we first introduce some notation.
\begin{definition}
 The gradient form $\Gamma$, associated with a $\mu$-Laplacian is defined by
\begin{align*}
2\Gamma(f,g)(x)& =(\Delta(f\cdot g)-f\cdot \Delta(g)-\Delta(f)\cdot g)(x)\\
               & =\asum{y\sim x}(f(y)-f(x))(g(y)-g(x)).
\end{align*}
We write $\Gamma(f)=\Gamma(f,f)$.
\end{definition}
Similarly,
\begin{definition}
The iterated gradient form $\Gamma_{2}$ is defined by
\[ 2\Gamma_{2}(f,g) = \Delta\Gamma(f,g)-\Gamma(f,\Delta g)-\Gamma(\Delta f,g). \]
We write $\Gamma_{2}(f)=\Gamma_{2}(f,f)$.
\end{definition}
\begin{definition}
The graph $G$ satisfies the CD inequality $CD(n,K)$ if, for any function $f$ and at every vertex $x \in V(G)$
\begin{equation}
\Gamma_{2}(f)\geq \frac{1}{n}(\Delta f)^{2}+K\Gamma(f). \label{eqn:cd}
\end{equation}
\end{definition}

On graphs -- where the Laplace operator fails to satisfy the chain rule -- satisfying the $CD(n,0)$ inequality seems insufficient to prove a generalization of \eqref{eq:ly}.  None the less, in~\cite{BHLLMY13} the authors prove a discrete analogue of the Li-Yau inequality.  The curvature notion they use is a modification of the standard curvature notion, which they call the exponential curvature dimension inequality.  In reality, the authors of \cite{BHLLMY13} introduce two slightly different curvature conditions, which they call $CDE$ and $CDE'$, both of which we recall below.

\begin{definition}
We say that a graph $G$ satisfies the {\it exponential curvature dimension inequality} $CDE(x,n,K)$ if  for any positive function $f : V\to \R^+$ such that $\Delta f(x) <0$, we have
\begin{equation} \widetilde{\Gamma_2}(f)(x)= \Gamma_2(f)(x)- \Gamma\left(f, \frac{\Gamma(f)}{f}\right)(x) \geq \frac{1}{n} (\Delta f)(x)^2 + K \Gamma(f)(x). \label{eqn:cde}
\end{equation}
We say that $CDE(n,K)$ is satisfied if $CDE(x,n,K)$ is satisfied for all $x \in V$.
\end{definition}

\begin{definition}
We say that a graph $G$ satisfies the $CDE'(x,n,K)$, if for any positive function $f : V\to \R^+$, we have
\begin{equation}
\widetilde{\Gamma_2}(f)(x) \geq \frac{1}{n} f(x)^2\left(\Delta \log f\right)(x)^2 + K \Gamma(f)(x).
\label{eqn:cde'}
\end{equation}
We say that $CDE'(n,K)$ is satisfied if $CDE'(x,n,K)$ is satisfied for all $x \in V$.
\end{definition}

The reason these are known as the {\it exponential} curvature dimension inequalities  is illustrated in Lemma 3.15 in \cite{BHLLMY13}, which states the following:
\begin{proposition}
If the semigroup generated by $\Delta$ is a diffusion semigroup (e.g. the Laplacian on a manifold), then $CD(n,K)$ and $CDE'(n,K)$ are equivalent.
\end{proposition}
To show that $CDE'(n,K) \Rightarrow CD(n,K)$ one takes an arbitrary function $f$, and applies \eqref{eqn:cde'} to $\exp(f)$ to verify that \eqref{eqn:cd} holds. Likewise, to verify that $CD(n,K) \Rightarrow CDE(n,K)$ one takes an arbitrary positive function $f$, and applies \eqref{eqn:cd} to $\log(f)$ to verify $\eqref{eqn:cde'}$.  This equivalence, however, makes strong use of the chain rule, and hence the fact that $\Delta$ generates a diffusion semigroup.

The relation between $CDE'(n,K)$ and $CDE(n,K)$ is the following:
\begin{remark}\label{rem:cde}
$CDE'(n,K)$ implies $CDE(n,K)$.
\end{remark}
\begin{proof}
Let $f:V \rightarrow \mathbb{R}^{+}$ be a positive function for which $\Delta f(x)< 0$. Since $\log s\leq s-1$ for all positive $s$, we can write
$$\Delta \log f(x)=\widetilde{\sum_{y\sim x}}\left(\log f(y)-\log f(x)\right)=\widetilde{\sum_{y\sim x}}\log \frac{f(y)}{f(x)}\leq \widetilde{\sum_{y\sim x}} \frac{f(y)-f(x)}{f(x)}=\frac{\Delta f(x)}{f(x)}< 0.$$
Hence squaring everything reverses the above inequality and we get
$$(\triangle f(x))^{2}\leq f(x)^{2}(\triangle \log f(x))^{2},$$
and thus  $CDE(x,n,K)$ is satisfied
$$\widetilde{\Gamma_2}(f)(x) \geq \frac{1}{n} f(x)^2\left(\triangle \log f\right)(x)^2 + K \Gamma(f)(x) > \frac{1}{n} (\Delta f)(x)^2 + K \Gamma(f)(x).$$
\end{proof}

In \cite{BHLLMY13}, the $CDE(n,K)$ inequality is preferred: the $\Delta \log(f)$ term occurring in the $CDE'$ inequality is awkward in the discrete case, the $CDE(n,K)$ inequality is weaker in general, and the $CDE(n,K)$ inequality sufficed for proving the Li-Yau inequality.

None the less, as the results in this paper will show, for the purposes of
applying semigroup arguments the $CDE'(n,K)$ inequality is to be preferred.  The primary reason for this is the fact that $CDE'(n,K)$ implies a non trivial lower bound on $\Gammatt(f)$ for a positive function $f$ at {\it every} point on a graph, as opposed to just the points where $\Delta f < 0$.  For maximum principle arguments, restricting to points where $\Delta f < 0$ turns out not to be a major restriction, but in the more global arguments we apply in this paper $CDE'(n,K)$ appears to be more useful.

We note that, in general, the conditions $CDE'$ and $CDE$ better capture the spirit of a Ricci curvature lower bound than the classical $CD$ condition.  For instance, every graph satisfies $CD(2,-1)$ -- that is, there is an absolute lower bound to the curvature of of graphs.  On the other hand, a $k$-regular tree satisfies $CDE(2,-d/2)$ and this negative curvature is (asymptotically) sharp.  Thus with the exponential curvature condition, negative curvature is unbounded.  This is unique amongst graph curvature notions.

Moreover, \cite{BHLLMY13} showed that lattices, and more generally Ricci-flat graphs in the sense of Chung and Yau~\cite{CY96} which include the abelian Cayley graphs, have non-negative curvature $CDE(n,0)$ and $CDE'(n,0)$.
We can show that the $CDE'(n,K)$ has product property(see also similar result for $CD(n,K)$  in \cite{LP14}).
So we can construct lot of graphs satisfy the $CDE'(n,0)$ assumption with different dimension $n$ by taking the Cartesian product of graphs which satisfying the  $CDE'(n,0)$.

\subsection{Main Results}

The first main result, alluded to in the introduction, is that satisfying $CDE'(n,0)$ is sufficient to imply that a graph satisfies several important conditions: volume doubling, the Poincar\'e inequality, Gaussian bounds for the heat kernel, and the continuous-time Harnack inequality.  For preciseness, we state these conditions now:
\begin{definition}{\hspace{0.1in}}\\

\begin{enumerate}
\item [($DV$)] A graph $G$ satisfies the {\bf volume doubling} property $DV(C)$ for constant $C > 0$ if for all $x \in V$ and all $r > 0$:
\[
V(x,2r) \leq C V(x,r).
\]
\item[($P$)] A graph $G$ satisfies the  {\bf Poincar\'{e} inequality} $P(C)$ for a constant $C > 0$ if
\[\sum_{x\in B(x_0,r)}m(x)|f(x)-f_B|^2\leq C r^2 \sum_{x,y\in B(x_0,2r)}\omega_{xy}(f(y)-f(x))^2,\]
for all $f\in V^{\mathbb{R}}$, for all $x_0\in V$,  and for all $r\in \mathbb{R^+}$, where
\[f_B=\frac{1}{V(x_0,r)}\sum_{x\in B(x_0,r)}m(x)f(x).\]
\item[($\mathcal{H}$)] Fix $\eta\in (0,1)$ and $0<\theta_1<\theta_2<\theta_3<\theta_4$ and $C>0$. $G$ satisfies the {\bf continuous-time Harnack inequality}  $\mathcal{H}(\eta,\theta_1,\theta_2,\theta_3,\theta_4,C)$, if for all $x_0\in V$ and $s, R\in \mathbb{R^{+}}$, and every positive solution $u(t,x)$ to the heat equation on $Q=[s,s+\theta_4R^{2}]\times B(x_0,R)$, we have
\[\sup_{Q^{-}}u(t,x)\leq C\inf_{Q^{+}}u(t,x),\]
where $Q^{-}=[s+\theta_1R^{2},s+\theta_2R^{2}]\times B(x_0, \eta R)$, and $Q^{+}=[s+\theta_3R^{2},s+\theta_4R^{2}]\times B(x_0, \eta R)$.
\item[($H$)]  Fix $\eta\in (0,1)$ and $0<\theta_1<\theta_2<\theta_3<\theta_4$ and $C>0$. $G$ satisfies the {\bf discrete-time Harnack inequality}  $H(\eta,\theta_1,\theta_2,\theta_3,\theta_4,C)$, if for all $x_0\in V$ and $s, R\in \mathbb{R^{+}}$, and every positive solution $u(x,t)$ to the heat equation on $Q=([s,s+\theta_4R^{2}]\cap \mathbb{Z})\times B(x_0,R)$, we have
\[(n^{-},x^{-})\in Q^{-},(n^{+},x^{+})\in Q^{+}, d(x^{-},x^{+})\leq n^{+}-n^{-}\]
implies
\[u(n^{-},x^{-})\leq C u(n^{+},x^{+}),\]
where $Q^{-}=([s+\theta_1R^{2},s+\theta_2R^{2}]\cap \mathbb{Z})\times B(x_0, \eta R)$, and $Q^{+}=([s+\theta_3R^{2},s+\theta_4R^{2}]\cap \mathbb{Z})\times B(x_0, \eta R)$.
\item[($G$)] Fix positive constants $c_{l}, C_{l}, C_{r}, c_{r}>0$.
The graph $G$ satisfies the {\bf Gaussian estimate} $G(c_{l},C_{l},C_{r},c_{r})$ if, whenever $d(x,y)\leq n$,
\[\frac{c_{l}m(y)}{V(x,\sqrt{n})}e^{-C_{l}\frac{d(x,y)^{2}}{n}}\leq p_{n}(x,y)\leq \frac{C_{r}m(y)}{V(x,\sqrt{n})}e^{-c_{r}\frac{d(x,y)^{2}}{n}}.\]
\end{enumerate}
\end{definition}


The following theorem is the first of the main results of this paper.

\begin{theorem}[cf. Theorem~\ref{maintheorem1}]\label{maintheorem-1}
If the graph satisfies $CDE'(n_{0},0)$ and $\Delta(\alpha)$, we have the following four properties.
\begin{enumerate}
  \item[$1)$]{} There exists $C_1, C_2, \alpha>0$ such that $DV(C_1)$, $P(C_2)$, and $\Delta(\alpha)$ are true.
  \item[$2)$]{} There exists $c_{l},C_{l},C_{r},c_{r}>0$ such that $G(c_{l},C_{l},C_{r},c_{r})$ is true.
  \item[$3)$]{} There exists $C_H$ such that $H(\eta,\theta_1,\theta_2,\theta_3,\theta_4,C_H)$ is true.
  \item[$3)'$]{} There exists $C_\mathcal{H}$ such that $\mathcal{H}(\eta,\theta_1,\theta_2,\theta_3,\theta_4,C_\mathcal{H})$ is true.
\end{enumerate}
\end{theorem}

A function $u$ on $G$ is called harmonic function if $\Delta u=0$. A harmonic function $u$ on $G$ has polynomial growth if there
is positive number $d$ such that

$$\exists x_{0}\in V, \exists C>0, \forall x\in V_{x_{0}}, \mid u(x)\mid\leq C d(x_{0},x)^{d}.$$

Combining Theorem~\ref{maintheorem-1} and Delmotte's Theorem 3.2 from \cite{D97}, we obtain the following result which confirms the analogue of Yau's conjecture
(\cite{Yau86}) on graphs.

\begin{theorem}
If the graph satisfies $CDE'(n_{0},0)$ and $\Delta(\alpha)$, then the dimension of space of harmonic functions on $G$ has polynomial growth
is finite.
\end{theorem}

Our final main result is the following Bonnet-Myers theorems for graphs. We defer the definition of canonical distance of graph until Section \ref{sec:DB}.

\begin{theorem}[cf. Theorem~\ref{th:intr-diam-bound} and  Theorem~\ref{diameterbound}]
Let $G=(V,E)$ be a locally finite, connected graph satisfying $CDE'(n,K)$, and $K>0$, then the diameter $\widetilde{D}$ of graph $G$ in terms of the canonical distance satisfies the inequality
\[\widetilde{D}\leq 4\sqrt{3}\pi \sqrt{\frac{n}{K}},\]
and in particular is finite.  Furthermore the diameter $D$ of graph $G$ in terms of the graph distance is also finite, and satisfies
\[D\leq 2\pi\sqrt{\frac{6D_\mu n}{K}}.\]

\end{theorem}

\section{A variational inequality, and Li-Yau type estimates} \label{sec:LY}
In this section we establish our main variational inequality which we develop in order to apply semigroup theoretic arguments in the non-diffusive graph case.  This is the content of Section \ref{sec:ve}.  Among the immediate applications of this variation inequality are a family of Li-Yau type inequalities which we derive in Section \ref{sec:lyly}.
\subsection{The heat kernel on graphs}
\subsubsection{The heat equation}
A function $u:[0,\infty) \times V \rightarrow \R$ is a positive solution to the {\bf heat equation} on $G=(V,E)$ if $u > 0$ at $u$ satisfies the differential equation
$$\Delta u=\partial_{t}u,$$
at every $x \in V$.

In this paper we are primarily interested in the {\bf heat kernel}, that is the fundamental solutions $p_t(x,y)$ of the heat equation.  These are defined so that for any bounded initial condition $u_0: V \to \mathbb{R}$, the function
\[
u(t,x) = \sum_{y \in V} \mu(y) p_t(x,y) u_0(y)~~~t>0, ~x \in V
\]
satisfies the heat equation, and $\lim_{t \to 0^{+}} u(t,x) = u_0(x)$

For any subset $U \subset V$, we denote by $\overset{\circ}{U}=\{x\in U: \forall y\sim x,~ y\in U\}$ the interior of $U$. The boundary of $U$ is $\partial U = U\setminus \overset{\circ}{U}$. We introduce the following version of the  maximum principle.
\begin{lemma}\label{lem:max-principle}
Let $U \subset V$ be finite and $T > 0$. Furthermore, assume that $u : [0,T] \times U \rightarrow \mathbb{R}$ is differentiable with respect to the first component and satisfies the inequality
$$\partial_{t} u \leq \Delta u$$
on $[0,T]\times \overset{\circ}{U}$.
Then, $u$ attains its maximum on the parabolic boundary
$$\partial_{P}([0,T]\times U)=(\{0\}\times U)\cup ([0,T]\times \partial U)$$
\end{lemma}
\begin{proof}
Suppose $u$ attains its maximum at a point $(t_{0}, x_{0})\in (0,T] \times U^{\circ}$ such that
\begin{equation} \partial_{t} u(t_0,x_0) < \Delta u(t_0,x_0) \label{eq:str} \end{equation}
Then
\begin{equation}
0 \leq \partial_t u(t_0,x_0) < \Delta u(t_{0}, x_{0})=\asum{y\sim x_{0}}\left(u(t_{0}, y)-u(t_{0}, x_{0})\right), \label{eqn:lapres}\end{equation} contradicting the maximality of $u$.

Otherwise, if at all $(t_0,x_0) \in (0,T] \times U^\circ$ which are maximum points at $u$, there is equality in \eqref{eq:str} we are done unless there is also equality in \ref{eqn:lapres}.  But this implies that $u$ is constant on $(0,T] \times U$, and hence there is a maximum point on the boundary as desired.
\end{proof}

\subsubsection{The heat equation an a domain}
Suppose $U \subset V$ is a finite subset of the vertex set of a graph. We consider the Dirichlet problem (DP),
\[\left\{
  \begin{array}{ll}
   \partial_{t}u(t,x)-\Delta_{U}u(t,x)=0, & \hbox{$x\in \overset{\circ}{U},t>0$,} \\
    u(0,x)=u_{0}(x), & \hbox{ $x\in \overset{\circ}{U}$,} \\
    u\mid_{[0,\infty)\times \partial U}=0.
  \end{array}
\right.\]
where $\Delta_{U}: \ell ^{2} (\overset{\circ}{U},\mu)\rightarrow \ell ^{2}(\overset{\circ}{U},\mu)$ denotes the Dirichlet Laplacian on $\overset{\circ}{U}$.

Note that $-\Delta_{U}$ is positive and self-adjoint, and $n:=\dim \ell ^{2} (\overset{\circ}{U},\mu)<\infty$. Thus the operator $-\Delta_U$ has eigenvalues  $0<\lambda_{1} \leq \lambda_2 < \cdots \leq \lambda_{n}$, along with an orthonormal set of eigenvectors $\phi_i$.  Here the orthonormality is with respect to the inner product with respect to the measure $\mu$, ie. $\langle \phi_i, \phi_j \rangle = \sum_{x \in V} \mu(x) \phi_i(x)\phi_j(x)$.

The operator $\Delta_{U}$ is a generator of the heat semigroup $P_{t,U}=e^{t\Delta_{U}}$,$t>0$.  Finite dimensionality makes the fact that  $e^{t\Delta_{U}}\phi_{i}=e^{-t\lambda_{i}}\phi_{i}$ transparent.  The heat kernel $p_{U}(t,x,y)$ for the finite subset $U$ is then given by
$$p_{U}(t,x,y)=P_{t,U}\frac{\delta_{y}}{\sqrt{\mu(y)}}(x),\quad \forall x,y\in \overset{\circ}{U}$$
where $\delta_{y}(x)=\sum_{i=1}^{n}\langle\phi_{i},\delta_{y}\rangle\phi_{i}(x)=\sum_{i=1}^{n}\phi_{i}(x)\phi_{i}(y)\sqrt{\mu(y)}$.
The heat kernel satisfies
$$ p_{U}(t,x,y)=\sum_{i=1}^{n}e^{-\lambda_{i}t}\phi_{i}(x)\phi_{i}(y),\quad \forall x,y\in \overset{\circ}{U}.$$

We record some useful properties of the heat kernel on a finite domain:
\begin{remark}\label{rem:finite-hkeanel-pro}
For $t,s>0$, $\forall x,y\in \overset{\circ}{U}$, we have
\begin{enumerate}
  \item $p_{U}(t,x,y)=p_{U}(t,y,x)$
  \item $p_{U}(t,x,y)\geq 0$,
  \item $\sum_{y\in \overset{\circ}{U}}\mu(y)p_{U}(t,x,y) \leq 1$,
  \item $\lim_{t\rightarrow 0^{+}}\sum_{y\in \overset{\circ}{U}}\mu(y)p_{U}(t,x,y) = 1$,
  \item $\partial_{t}p_{U}(t,x,y)=\Delta_{(U,y)}p_{U}(t,x,y)=\Delta_{(U,x)}p_{U}(t,x,y)$
  \item $\sum_{z\in \overset{\circ}{U}}\mu(z)p_{U}(t,x,z)p_{U}(s,z,y)=p_{U}(t+s, x,y)$
\end{enumerate}
\end{remark}
\begin{proof}
(1) and (5) follow from the above fact about the heat kernel, (2) and (3) are immediate consequences of the maximum principle.  Note that (4) follows from the continuity of the semigroup $e^{t\Delta}$ at $t = 0$, if the limit is understood in the $\ell^2$ sense.   As $U$ is finite all norms are equivalent and pointwise convergence follows also. (6) is easy to calculate in $\ell^2$,and it is called the semigroup property of heat kernel.
\end{proof}

\subsubsection{Heat equation on a infinite graph}

The heat kernel for an infinite graph can be constructed and its basic properties can be derived using the above ideas by taking an {\it exhaustion} of the graph.   An exhaustion of $G$ is a sequence $(U_k)$ of subsets of $V$, such that $U_{k}\subset \overset{\circ}{U}_{k+1}$ and $\cup_{k \in \mathbb{N}}U_{k}=V$. For any connected, countable graph $G$ such a sequence exists.  One may, for instance, fix a vertex $x_0 \in V$ take the sequence $U_{k} = B_{k}(x_{0})$ of metric balls of radius $k$ around $x_0$.  The connectedness of our graph $G$ implies that the union of these $U_k$ equals $V$.

Denoting by $p_{k}$, the heat kernel $p_{U_{k}}$ on $U_{k}$, we may extend $p_k$ to all of $(0,\infty) \times V \times V$,
\[p_{k}(t,x,y)=
\left\{
  \begin{array}{ll}
     p_{U_{k}}(t,x,y), & \hbox{$x,y\in \overset{\circ}{U_{k}}$;} \\
     0, & \hbox{o.w.}
  \end{array}
\right.
\]
Then for any $t>0$,and $x,y \in V,$ we let
$$p(t,x,y)=\lim_{k\rightarrow \infty}p_{k}(t,x,y).$$
The maximum principle implies the monotonicity of the heat kernels, i.e.
$p_{k}\leq p_{k+1}$, so the above limit exists (but could a priori be infinite).  Similarly, it is not a priori clear that $p$ is independent of the exhaustion chosen.  None the less, the limit is finite and independent of the exhaustion and $p$ is the desired heat kernel.    This construction is carried out in \cite{WE10} and \cite{WO09} for unweighted graphs, where the measure $\mu \equiv 1$.  For the general case, we refer to \cite{KL12}.

For convenience, we record some important properties of the heat kernel $p$ which we will use in the paper.

\begin{remark}\label{rem:heatkpro}
For $t,s>0$, $\forall x,y\in V$, we have
\begin{enumerate}
  \item $p(t,x,y)=p(t,y,x)$
  \item $p(t,x,y)\geq 0$,
  \item $\sum_{y\in V}\mu(y)p(t,x,y) \leq 1$,
  \item $\lim_{t\rightarrow 0^{+}}\sum_{y\in V}\mu(y)p(t,x,y) = 1$,
  \item $\partial_{t}p(t,x,y)=\Delta_yp(t,x,y)=\Delta_xp(t,x,y)$
  \item $\sum_{z\in V}\mu(z)p(t,x,z)p(s,z,y)=p(t+s,x,y)$
\end{enumerate}
\end{remark}

 From here, the semigroup $P_{t}:V^\mathbb{R}\rightarrow V^\mathbb{R}$
acting on bounded functions $f: V \to \mathbb{R}$ as follows.
  for any bounded function $f\in V^\mathbb{R}$,
\[ P_{t}f(x)=\lim_{k\rightarrow \infty}\sum_{y\in V}\mu(y)p_{k}(t,x,y)f(y)=\sum_{y\in V}\mu(y)p(t,x,y)f(y)
\]
where $\lim_{t\rightarrow 0^{+}}P_{t}f(x)=f(x)$, and $P_{t}f(x)$ is a solution of the heat equation. From the properties of the heat kernel, and the boundedness of $f$, that is, there exists a constant $C>0$, such that for any $x\in V$, $\sup_{x\in V}|f(x)|\leq C$, we have
\[\left |\sum_{y\in V}\mu(y)p(t,x,y)f(y)\right|\leq C\lim_{k\rightarrow \infty}\sum_{y\in V}\mu(y)p_{k}(t,x,y)\leq C<\infty,\]
so the semigroup is well-defined. Besides, the different definitions of the heat semigroup coincide when $\Delta$ is a bounded operator or in finite graphs, that is
\[P_tf(x)=e^{t\Delta}f(x)=\sum_{k=0}^{+\infty} \frac{t^k\Delta^k}{k!}f(x)=\sum_{y\in V}\mu(y)p(t,x,y)f(y).\]

Again we record, without proof, some well known but useful properties of the semigroup $P_t$.
\begin{proposition}\label{pro:semigroup}
For any bounded function $f,g\in V^{\mathbb{R}}$, and $t,s>0$, for any $x\in V$,
\begin{enumerate}
  \item If $0 \leq f(x) \leq 1$, then $0 \leq P_{t}f(x) \leq 1$,
  \item $P_{t}\circ P_{s}f(x)=P_{t+s}f(x)$,
  \item $\Delta P_{t}f(x)=P_{t} \Delta f(x)$.
\end{enumerate}
\end{proposition}

\subsection{The main variational inequality} \label{sec:ve}
Finiteness of $D_\mu$ implies boundedness of the operators $\Delta$ and $\Gamma$.  We in turn derive the following Lemma:

\begin{lemma}\label{lem:bound}  Suppose $G=(V,E)$ is a (finite or infinite) graph satisfying the condition $CDE'(n,K)$.  Then, for a positive and bounded  solution $u(t,x)$ to the heat equation on $G$, the function
$\frac{\Delta u}{2\sqrt{u}}$ on $G$ is bounded at all $t > 0$.
\end{lemma}
\begin{proof}
The statement is obvious for finite graphs $G$, so we restrict our attention to infinite graphs.

Fix $R \in \mathbb{N}$ and vertex $x_0 \in V$.   We define a cutoff function $\varphi$ by letting
\begin{equation*}
\varphi(x)=\left\{
               \begin{array}{ll}
                 0, & d(x,x_0)>2R \\
                 \frac{2R-d(x,x_0)}{R}, & R\leq d(x,x_0)\leq2R\\
                 1, & d(x,x_0)<R
               \end{array}
             \right.
\end{equation*}

Let
\[F=\varphi\cdot\frac{\Gamma(\sqrt{u})}{\sqrt{u}},\]

It is easy to observe that, as $0 \leq\varphi(x)\leq 1$ for any $x \in V$,  $|\Delta \varphi| \leq 2D_\mu$.  As $u$ is bounded, there exists constants $c_1, c_2$ so that $0\leq\Gamma(\sqrt{u})\leq c_1$, and $|\Gamma(\Gamma(\sqrt{u}),\varphi)|\leq c_2$ as well.

Fix an arbitrary $T>0$, let $(x^*,t^*)$ be a maximum point of $F$ in $V\times [0,T]$.  Clearly such a maximum exists, as $F \geq 0$ and $F$ is only positive on a bounded region.   We may assume $F(x^*,t^*)>0$. In what follows all computations take place at the point $(x^*,t^*)$. Let $\mathcal{L}=\Delta-\partial_t$, we apply Lemma 4.1 in \cite{BHLLMY13} with the choice of $g=u$. This gives
\[\mathcal{L}(\sqrt{u}F)\leq \mathcal{L}(\sqrt{u})F=-\frac{F^2}{\varphi},\]
and since for any $x\in V$, we have
\[
\begin{split}
\partial_{t}\Gamma(\sqrt{u})(x)
&= \partial_{t}\frac{1}{2}\widetilde{\sum_{y\sim x}}\left(\sqrt{u}(y)-\sqrt{u}(x)\right)^{2}\\
&=\widetilde{\sum_{y\sim x}}(\sqrt{u}(y)-\sqrt{u}(x)) (\partial_{t}\sqrt{u}(y)-\partial_{t}\sqrt{u}(x))\\
&= 2 \Gamma(\sqrt{u},\frac{\Delta u}{2\sqrt{u}})(x),
\end{split}
\]
then
\[\mathcal{L}(\sqrt{u}F)=\mathcal{L}(\varphi\cdot \Gamma(\sqrt{u}))=\Delta \varphi\cdot \Gamma(\sqrt{u})+2\Gamma(\Gamma(\sqrt{u}),\varphi)+2\varphi\cdot\widetilde{\Gamma}_2(\sqrt{u}).\]
Applying the $CDE'(n,K)$ condition  and throwing away the $\frac{1}{n}u(\Delta\log\sqrt{u})^2$ term, we obtain
\[-\frac{F^2}{\varphi}\geq \Delta \varphi\cdot \Gamma(\sqrt{u})+2\Gamma(\Gamma(\sqrt{u}),\varphi)+2\varphi K\Gamma(\sqrt{u}).\]
From here, we conclude that
\[F^2(x^*,t^*)\leq 2(D_\mu+|K|)c_1+c_2,\]
Thus there exists some $C>0$ so that
\[F(x^*,t^*)\leq C.\]
For $x\in B(x_0,R)$,
\[\frac{\Gamma(\sqrt{u})}{\sqrt{u}}(T,x)=F(x,T)\leq F(x^*,t^*)\leq C.\]
From the equation $\Delta u=2\sqrt{u}\Delta\sqrt{u}+2\Gamma(\sqrt{u})$, we can obtain $\frac{\Delta u}{2\sqrt{u}}$ is bounded at any positive $T > 0$ as well.
\end{proof}
Thus for any bounded function $0<f\in \ell^{\infty}(V,\mu)$ on $G(V,E)$, the function $\Gamma(\sqrt{P_{T-t} f})$, for any $0\leq t<T$  is likewise bounded.

Given a positive bounded $f$, we introduce the function
$$\phi(t,x)=P_{t}(\Gamma(\sqrt{P_{T-t} f}))(x),\quad 0\leq t < T, x\in V.$$

From here we obtain the following (rather crucial) result.
\begin{lemma}\label{lem:derivation}
Suppose that $G$ satisfies the condition $CDE'(n,K)$.  Then, for every $0\leq t < T$, any $x\in V$, the function $\phi$ satisfies
$$\partial_{t}\phi(t,x)=2 P_{t}(\widetilde{\Gamma}_{2}(\sqrt{P_{T-t} f}))(x).$$
\end{lemma}

\begin{proof}
For any $x\in V$,
\[
\begin{split}
\partial_{t}P_{t}(\Gamma(\sqrt{P_{T-t} f}))(x)
        & =\partial_{t}\left(\sum_{y\in V}\mu(y)p(t,x,y) \Gamma(\sqrt{P_{T-t} f})(y)\right)\\
        & =\sum_{y\in V} \mu(y)\left(\Delta p(t,x,y) \Gamma(\sqrt{P_{T-t} f})(y) + p(t,x,y) \partial_{t}\Gamma(\sqrt{P_{T-t} f})(y) \right)\\
        & =\sum_{y\in V} \mu(y)\left(\Delta p(t,x,y) \Gamma(\sqrt{P_{T-t} f})(y) - 2p(t,x,y)\Gamma(\sqrt{P_{T-t} f},\frac{\Delta P_{T-t} f}{2\sqrt{P_{T-t} f}})(y)\right)\\
        & =\sum_{y\in V} \mu(y)p(t,x,y)\left( \Delta\Gamma(\sqrt{P_{T-t} f})(y)  - 2\Gamma(\sqrt{P_{T-t} f},\frac{\Delta P_{T-t} f}{2\sqrt{P_{T-t} f}})(y)\right)\\
        & =2 P_{t}(\widetilde{\Gamma}_{2}(\sqrt{P_{T-t} f}))(x)
\end{split}
\]

For the third equality, we observe that for any $x\in V$,
\[
\begin{split}
\partial_{t}\Gamma(\sqrt{P_{T-t} f})(x)
&= \partial_{t}\frac{1}{2}\widetilde{\sum_{y\sim x}}\left(\sqrt{P_{T-t} f}(y)-\sqrt{P_{T-t} f}(x)\right)^{2}\\
&=\widetilde{\sum_{y\sim x}}(\sqrt{P_{T-t} f}(y)-\sqrt{P_{T-t} f}(x)) (\partial_{t}\sqrt{P_{T-t} f}(y)-\partial_{t}\sqrt{P_{T-t} f}(x))\\
&= 2 \Gamma(\sqrt{P_{T-t} f},\partial_{t}\sqrt{P_{T-t} f})(x),
\end{split}
\]
and,
$$\partial_{t}\sqrt{P_{T-t} f}=\frac{\partial_{t}P_{T-t} f}{2\sqrt{P_{T-t} f}} =-\frac{\Delta P_{T-t} f}{2\sqrt{P_{T-t} f}},$$
where $\partial_{t}P_{T-t} f=-\Delta P_{T-t} f$.

In the fourth step, note that due to the boundedness of $f$, the function $\Delta\Gamma(\sqrt{P_{T-t} f})$ is likewise bounded.  Similarly from Lemma~\ref{lem:bound}, $\Gamma(\sqrt{P_{T-t} f},\frac{\Delta P_{T-t} f}{2\sqrt{P_{T-t} f}})$ is bounded as well. Like the proof of Proposition~\ref{pro:semigroup}, we have
\[
\begin{split}
&\sum_{y\in V} \mu(y)\left(\Delta p(t,x,y) \Gamma(\sqrt{P_{T-t} f})(y) - 2p(t,x,y)\Gamma(\sqrt{P_{T-t} f},\frac{\Delta P_{T-t} f}{2\sqrt{P_{T-t} f}})(y)\right)\\
&=\sum_{y\in V} \mu(y)\Delta p(t,x,y) \Gamma(\sqrt{P_{T-t} f})(y) -\sum_{y\in V} \mu(y) 2p(t,x,y)\Gamma(\sqrt{P_{T-t} f},\frac{\Delta P_{T-t} f}{2\sqrt{P_{T-t} f}})(y)\\
&=\sum_{y\in V} \mu(y) p(t,x,y) \Delta\Gamma(\sqrt{P_{T-t} f})(y) -\sum_{y\in V} \mu(y) 2p(t,x,y)\Gamma(\sqrt{P_{T-t} f},\frac{\Delta P_{T-t} f}{2\sqrt{P_{T-t} f}})(y)\\
& =\sum_{y\in V} \mu(y)p(t,x,y)\left( \Delta\Gamma(\sqrt{P_{T-t} f})(y)  - 2\Gamma(\sqrt{P_{T-t} f},\frac{\Delta P_{T-t} f}{2\sqrt{P_{T-t} f}})(y)\right),
\end{split}
\]
where various interchanges of sums is justified due to the boundedness of the terms multiplied by the heat kernel (and hence absolute convergence of the sums).

An important problem that need be solved at last is the exchange between summation and derivation in the second step. To that end, first we know the different definitions of the heat semigroup coincide since $\Delta$ is a bounded operator, that is
\[P_tf(x)=e^{t\Delta}f(x)=\sum_{k=0}^{+\infty} \frac{t^k\Delta^k}{k!}f(x)=\sum_{y\in V}\mu(y)p(t,x,y)f(y).\]
Then consider another function $P_tf^t(x)$, where $f^t(x)$ is a positive function with respect to $t\in[0,+\infty)$ and $x\in V$ (here $f^t(x)=2\widetilde{\Gamma}_{2}(\sqrt{P_{T-t} f})(x)$). These above summations have a nice convergency when $f^t(x)$ is a uniform bounded function. Since there exists a constant $C>0$ such that $\mid f^t(x)\mid\leq C$ for any $t\in [0,T]$, we have
\[|\Delta f^t(x)|=|\widetilde{\sum_{y\sim x}}(f^t(y)-f^t(x))|\leq 2D_\mu C,\]
and for iteration, we obtain for any $k\in \mathbf{N}_{\geq0}$ and $x\in V$,
\[|\Delta^k f^t(x)(x)|\leq2^kD_\mu^k C.\]
And
\[\sum_{k=0}^{+\infty}\frac{T^k}{k!}2^kD_\mu^k C=C e^{2D_\mu T}<\infty.\]
Therefore, the series
\[P_tf^t(x)=\sum_{y\in V}\mu(y)p(t,x,y)f^t(y)=\sum_{k=0}^{+\infty} \frac{t^k\Delta^k}{k!}f^t(x)\]
converges uniformly on $[0,T]$.

This ends the proof of Lemma~\ref{lem:derivation}.
\end{proof}

We now obtain some graph theoretical analogues to theorems of Baudoin and Garofalo \cite{BG$09$} originating on in the manifold setting.  In some sense our main observation is that the $CDE'(n,K)$ condition can be used in order to overcome the diffusive semigroup assumption usually needed for arguments involving the heat semigroup.  This is one of the primary places where we note that the $CDE(n,K)$ condition favored in \cite{BHLLMY13} is seemingly insufficient to prove the result.

\begin{theorem}\label{th:main-var-ineq}
Let $G=(V,E)$ be a locally finite, connected graph satisfying $CDE'(n,K)$, then for every $\alpha:[0,T] \to \mathbb{R^{+}}$ be a smooth and positive function and non-positive smooth function $\gamma : [0,T] \to \mathbb{R}$, we have for any positive and bounded function $f$
\begin{equation}
\partial_{t}(\alpha \phi) \geq (\alpha'-\frac{4\alpha\gamma}{n}+2\alpha K)\phi+\frac{2\alpha\gamma}{n}\Delta P_{T} f-\frac{2\alpha\gamma^{2}}{n} P_{T}f.
\label{equ:main-var-ine}
\end{equation}
\end{theorem}

\begin{proof}
For any $x\in V$, we have
\begin{equation}
\begin{split}
\partial_{t}(\alpha \phi)(x)
&= \alpha'\phi(x)+ 2\alpha P_{t}(\widetilde{\Gamma}_{2}(\sqrt{P_{T-t} f}))(x)\\
&\geq \alpha'\phi(x)+ 2\alpha P_{t}\left(\frac{1}{n}\left(\sqrt{P_{T-t} f}\Delta \log\sqrt{P_{T-t} f}\right)^{2}+K \Gamma(\sqrt{P_{T-t} f})\right)(x)\\
&\geq (\alpha'+2 \alpha K)\phi(x)+ 2\alpha \sum_{\substack{y \sim x\\ \Delta\sqrt{P_{T-t} f}(y)< 0}}\mu(y)p(t,x,y)\frac{1}{n}\left(\Delta\sqrt{P_{T-t} f}\right)^{2}(y)\\
&\hspace{0.3in}+2\alpha \sum_{\substack{y \sim x \\ \Delta\sqrt{P_{T-t} f}(y)\geq 0}}\mu(y)p(t,x,y)\frac{1}{n}\left(\sqrt{P_{T-t} f}\Delta\log\sqrt{P_{T-t} f}\right)^{2}(y)\\
&\geq (\alpha'+2 \alpha K)\phi(x)+ \frac{2\alpha}{n}P_{t}(\gamma \Delta P_{T-t}f - 2\gamma \Gamma(\sqrt{P_{T-t} f}) - \gamma^{2}P_{T-t}f)(x)\\
&= (\alpha'+2 \alpha K)\phi(x)+ \frac{2\alpha \gamma}{n} P_{t}(\Delta P_{T-t}f)(x)- \frac{4\alpha \gamma}{n} P_{t}(\Gamma(\sqrt{P_{T-t} f}))(x)-\frac{2\alpha \gamma^{2}}{n}P_{t}(P_{T-t}f)(x)\\
&= (\alpha'-\frac{4\alpha\gamma}{n}+2\alpha K)\phi(x)+\frac{2\alpha\gamma}{n}\Delta P_{T} f(x)-\frac{2\alpha\gamma^{2}}{n} P_{T} f(x).
\end{split}
\notag
\end{equation}
The first inequality in the above proof comes from applying the $CDE'(n,K)$ inequality to $\sqrt{P_{T-t} f}$.  The second one comes from Jensen's inequality, under the assumption that $(\Delta \sqrt{P_{T-t} f})(y) < 0$.  This is essentially the contents of Remark~\ref{rem:cde} -- really we  apply the $CDE(n,K)$ inequality at points so that $\Delta \sqrt{P_{T-t} f}(y)<0$.

The third inequality is a bit more subtle and is derived as follows:  Clearly  for any function $\gamma$, one has
$$(\Delta\sqrt{P_{T-t} f})(y)^{2} \geq 2\gamma \sqrt{P_{T-t} f}(y)\Delta\sqrt{P_{T-t} f}(y)-\gamma^{2}P_{T-t} f(y).$$

Since $\gamma$ is non-positive, if $\Delta\sqrt{P_{T-t} f}(y)\geq 0$ the right hand of the above inequality is also non-positive.  Thus in this case it is also true that
$$\left(\sqrt{P_{T-t} f}\triangle\log\sqrt{P_{T-t} f}\right)^{2}(y) \geq 2\gamma \sqrt{P_{T-t} f}(y)\triangle\sqrt{P_{T-t} f}(y)-\gamma^{2}P_{T-t} f(y),$$
as the left hand side of this inequality is clearly non-negative.

Furthermore, by the identity $\Delta u = 2\sqrt{u} \Delta \sqrt{u} + 2\Gamma(u)$,
$$2\sqrt{P_{T-t} f}\Delta\sqrt{P_{T-t} f}= \Delta P_{T-t} f-2\Gamma(\sqrt{P_{T-t} f}),$$
Therefore,
\begin{multline*}
\sum_{\substack{y \sim x \\\Delta\sqrt{P_{T-t} f}(y)<0}}\mu(y)p(t,x,y)\left(\Delta\sqrt{P_{T-t} f}\right)^{2}(y) \\+\sum_{\substack{y \sim x \\ \Delta\sqrt{P_{T-t} f}(y)\geq 0}}\mu(y)p(t,x,y)P_{T-t} f(y)\left(\Delta\log\sqrt{P_{T-t} f}\right)^{2}(y) \\
\geq P_{t}(\gamma \Delta P_{T-t}f - 2\gamma \Gamma(\sqrt{P_{T-t} f}) - \gamma^{2}P_{T-t}f)(x),
\notag
\end{multline*}
as desired.
\end{proof}

\subsection{Li-Yau inequalities} \label{sec:lyly}

The precise power of Theorem \ref{th:main-var-ineq} is, perhaps, a bit hard to appreciate at first.  As an application, it can be used to give an alternative derivation of the Li-Yau inequality.  Indeed, it can be used to derive a {\it family} of similar differential Harnack inequalities.  The key in applying Theorem \ref{th:main-var-ineq} is to choose $\gamma$ in a careful way so that the things simplifying carefully.

For instance, suppose for some (smooth) function $\alpha$ we choose $\gamma$ in such a way so that
$$\alpha'-\frac{4\alpha\gamma}{n}+2\alpha K=0.$$
That is, choose
$$\gamma=\frac{n}{4}\left(\frac{\alpha'}{\alpha}+2K\right).$$

If $\alpha$ is chosen appropriately to make $\gamma$ non-positive, then we may integrate the inequality \eqref{equ:main-var-ine} obtained in Theorem \ref{th:main-var-ineq} from $0$ to $T$, and obtain an estimate.  If we denote $W =\sqrt{\alpha}$, we obtain the following result.

\begin{theorem}
Let $G=(V,E)$ be a locally finite and connected graph satisfying $CDE'(n,K)$, and $W : [0, T] \to \mathbb{R^{+}}$ be a smooth function such that
$$W(0)=1,W(T)=0,$$
and so that
\[
W'(t) \leq -K W(t)
\]
for $0 \leq t \leq T.$
Then for any bounded and positive function $f\in V^{\mathbb{R}}$, we have
\begin{equation}
\begin{split}
\frac{\Gamma(\sqrt{P_{T} f})}{P_{T} f}
& \leq \frac{1}{2} \left(1-2K \int_{0}^T W(s)^{2} ds\right) \frac{\Delta P_{T}f}{P_{T}f}\\
& +\frac{n}{2}\left(\int_{0}^T W'(s)^{2} ds + K^{2}\int_{0}^T W(s)^{2} ds-K \right).
\end{split}
\label{equ:ly-family}
\end{equation}
\end{theorem}
Here, the condition $W' \leq -KW$ amounts to the non-positivity of $\gamma$.   As observed in \cite{BG$09$}, the family obtained by taking
$$W(t)=\left(1-\frac{t}{T}\right)^{b},$$
for any $b > \frac{1}{2}$ is quite interesting in the regime where $-\frac{b}{T} < K$.

For this family,
$$\int_{0}^T W(s)^{2} ds=\frac{T}{2b+1},$$
and
$$\int_{0}^T W'(s)^{2} ds=\frac{b^{2}}{(2b-1)T}.$$

Thus for such a choice of $W$, the estimate \eqref{equ:ly-family} yields
\begin{equation}
\frac{\Gamma(\sqrt{P_{T} f})}{P_{T} f} \leq \frac{1}{2} \left(1-\frac{2K T}{2b+1}\right) \frac{\Delta P_{T}f}{P_{T}f} +\frac{n}{2}\left(\frac{b^{2}}{(2b-1)T} +\frac{K^{2}T}{2b+1}-K \right).
\label{equ:ly-general}
\end{equation}
When $K = 0$ and $b=1$, this reduces to the familiar Li-Yau inequality on graphs (as derived by \cite{BHLLMY13}).   Indeed, per the identity $\Delta P_{t}f=\partial_{t}P_{t} f=2\sqrt{P_{t} f}\partial_{t}\sqrt{P_{t} f}$ and switching the notion $T$ to $t$, \eqref{equ:ly-general} reduces to:
\begin{equation*}
\frac{\Gamma(\sqrt{P_{t} f})}{P_{t} f} -\frac{\partial_{t}\sqrt{P_{t} f}}{\sqrt{P_{t} f}} \leq \frac{n}{2t}, \quad  t>0.
\end{equation*}

\section{Volume Growth}  \label{sec:VG}

While the Li-Yau inequality is an attractive consequence of Theorem \ref{th:main-var-ineq}, a version was already known to hold on graphs using the $CDE(n,K)$ curvature dimensional inequality (which is slightly weaker than the $CDE'(n,K)$ inequality used in Theorem \ref{th:main-var-ineq}.).

In this section, we begin by exhibiting a further application of the variational inequality, and use it derive volume doubling from non-negative curvature, which was out of reach from previous results.
\begin{theorem}\label{th:exp-ineq}
Let $G=(V,E)$ be a locally finite and connected graph satisfying $CDE'(n,0)$, there exists an absolute positive constant $\rho > 0$, and $A > 0$, depending only on $n$, such that
\begin{equation}
P_{Ar^{2}} \left(\mathbf{1}_{B(x,r)}\right)(x)\geq \rho, \quad x\in V,\quad r>\frac{1}{2}.
\label{equ:exp-ineq}
\end{equation}
\end{theorem}
\begin{proof}
Again, we proceed by carefully choosing a $\gamma$ to apply Theorem \ref{th:main-var-ineq}.  Let
$$\alpha(t)=\tau+T-t,$$
$$\gamma(t)=-\frac{n}{4(\tau+T-t)},$$
for $\tau > 0$, and $K=0$. For such a choice
$$\alpha'-\frac{4\alpha\gamma}{n}+2\alpha K=0,\quad \frac{2\alpha\gamma}{n}=-\frac{1}{2},\quad
\frac{2\alpha\gamma^{2}}{n}=\frac{n}{8(\tau+T-t)},$$
again simplifying the main inequality.  We integrate the inequality from $0$ to $T$, obtaining
\begin{equation}
\tau P_{T}(\Gamma(\sqrt{f}))-(T+\tau)\Gamma(\sqrt{P_{T} f})\geq -\frac{T}{2}\Delta P_{T}f-\frac{n}{8}\log\left(1+\frac{T}{\tau}\right)P_{T}f.
\label{equ:int-main-var1}
\end{equation}

Now, suppose $f$ is a non-positive $c$-Lipschitz function (that is, $|f(y)-f(x)| \leq c$ if $x \sim y$.)  Fix $\lambda \geq 0$, and consider the
function $\varphi = e^{2\lambda f}$.  Clearly, $f$ is positive and bounded.  Let
\[
\psi(\lambda,t) = \frac{1}{2\lambda} \log (P_t e^{2\lambda f}),
\]
so that
\[
P_t \varphi = P_t(e^{2\lambda f}) = e^{2\lambda \psi}.
\]

Applying  \eqref{equ:int-main-var1} to $\varphi$, and switching notation from $T$ to $t$, one obtains that
\begin{equation} \tau P_{t}(\Gamma(e^{\lambda f}))-(t+\tau)\Gamma(e^{\lambda \psi})\geq -\frac{t}{2}\Delta P_{t} \varphi-\frac{n}{8}\log\left(1+\frac{t}{\tau}\right)e^{2\lambda \psi}.
\label{eq:ineq}
\end{equation}
Fix $x\in V$.  Taking $C(\lambda,c)=\sqrt{\frac{D_{\mu}}{2}}c e^{\lambda c}<\infty$, we have
\[
\begin{split}
\Gamma(e^{\lambda f})(x)
&=\frac{1}{2}\widetilde{\sum_{y\sim x}}\left(e^{\lambda f(y)}-e^{\lambda f(x)}\right)^{2}\\
&=\frac{1}{2}e^{2\lambda f(x)}\widetilde{\sum_{y\sim x}}\left(e^{\lambda (f(y)-f(x))}-1\right)^{2}\\
&=\frac{1}{2}e^{2\lambda f(x)}\left(\widetilde{\sum_{0\leq f(y)-f(x)\leq c}}\left(e^{\lambda (f(y)-f(x))}-1\right)^{2}+\widetilde{\sum_{-c\leq f(y)-f(x)\leq 0}}\left(e^{\lambda (f(y)-f(x))}-1\right)^{2}\right)\\
&\leq \frac{1}{2}e^{2\lambda f(x)}\left(e^{2\lambda c}\widetilde{\sum_{0\leq f(y)-f(x)\leq c}}\left(1-e^{-\lambda c}\right)^{2}+\widetilde{\sum_{-c\leq f(y)-f(x)\leq 0}}\left(e^{-\lambda c}-1\right)^{2}\right)\\
&\leq \frac{1}{2}e^{2\lambda f(x)}e^{2\lambda c}\widetilde{\sum_{y\sim x}}\left(e^{-\lambda c}-1\right)^{2}\\
&\leq C(\lambda,c)^{2}\lambda^{2}e^{2\lambda f(x)}.
\end{split}\]

This enables us to upper bound the left hand side of \eqref{eq:ineq}, obtaining
$$\tau P_{t}(\Gamma(e^{\lambda f}))-(t+\tau)\Gamma(e^{\lambda \psi}) \leq \tau P_{t}(\Gamma(e^{\lambda f}))\leq C(\lambda,c)^{2}\lambda^{2}\tau P_{t}(e^{2\lambda f})= C(\lambda,c)^{2}\lambda^{2}\tau e^{2\lambda \psi}.$$

Combining this with the fact that
$$\Delta P_{t} \varphi=\partial_{t} e^{2\lambda \psi}=2\lambda e^{2\lambda \psi}\partial_{t}\psi,$$
we obtain that
\begin{equation}
\partial_{t}\psi \geq -\frac{\lambda}{t}\left(C(\lambda,c)^{2}\tau+\frac{n}{8\lambda^{2}}\log(1+\frac{t}{\tau})\right).
\label{equ:partial-ineq}
\end{equation}

Since \eqref{equ:partial-ineq} holds for all $\tau$, we optimize.  Setting $\tau$ to be the optimal value,
$$\tau_{0}=\frac{t}{2}\left(\sqrt{1+\frac{n}{2\lambda^{2}C(\lambda,c)^{2}t}}-1\right),$$
and substituting into  \eqref{equ:partial-ineq} obtain
\begin{equation}
-\partial_{t}\psi \leq \lambda C(\lambda,c)^{2} G\left(\frac{1}{\lambda^{2}C(\lambda,c)^{2}t}\right).
\label{equ:partial-ineq2}
\end{equation}
Here,
$$G(s)=\frac{1}{2}\left(\sqrt{1+\frac{n}{2}s}-1\right)+\frac{n}{8}s\log\left(1+\frac{2}{\sqrt{1+\frac{n}{2}s}-1}\right),\quad s>0.$$

Note that $G(s)\to 0$ as $s \to 0^{+}$, and that $G(s) \sim \sqrt{\frac{ns}{2}}$ as $s \to +\infty$.   Integrate the inequality \eqref{equ:partial-ineq2} between $t_{1}$ and $t_{2}$ (for $t_1 \leq t_2$) and we obtain that,
$$\psi(\lambda,t_{1})\leq \psi (\lambda,t_{2})+\lambda C(\lambda,c)^{2} \int_{t_{1}}^{t_{2}} G\left(\frac{1}{\lambda^{2}C(\lambda,c)^{2} t}\right)dt.$$
Jensen's inequality in $\psi$ yields that
$$2\lambda \psi(\lambda,t)= \ln(P_{t}e^{2\lambda f})\geq P_{t}(\ln e^{2\lambda f})=2\lambda P_{t}f.$$
This yields that $\lambda P_{t_{1}}f \leq \lambda \psi(\lambda,t_{1})$, and combining with the previous inequality we have that for all $t_1 \leq t_2$.
$$P_{t_{1}}(\lambda f)\leq \lambda\psi (\lambda,t_{2})+\lambda^{2}C(\lambda,c)^{2}\int_{t_{1}}^{t_{2}} G\left(\frac{1}{\lambda^{2}C(\lambda,c)^{2}t}\right)dt.$$
Replacing $t_2$ with $t$, and letting $t_1 \to 0^+$ we obtain
\begin{equation}
\lambda f\leq \lambda\psi (\lambda,t)+\lambda^{2}C(\lambda,c)^{2}\int_{0}^{t} G\left(\frac{1}{\lambda^{2}C(\lambda,c)^{2}\tau}\right)d\tau.
\label{equ:G-ineq}
\end{equation}

Now fix a vertex  $x\in V$.  Let $B = B(x, r)$, and consider the function $f(y)=-d(y,x)$. Clearly $f$ is $1$-Lipschitz.  For such a $1$-Lipschitz function, we may use $C(\lambda,c) =\sqrt{\frac{D_{\mu}}{2}}e^{\lambda}$ in the proceeding.

Clearly,
$$e^{2\lambda f}\leq e^{-2\lambda r}\mathbf{1}_{B^{c}}+\mathbf{1}_{B}.$$

Thus for every $t > 0$ one has,
$$e^{2\lambda\psi(\lambda,t)(x)}=P_{t}(e^{2\lambda f})(x)\leq e^{-2\lambda r}+P_{t}(\mathbf{1}_{B})(x)$$
and we obtain the lower bound
$$P_{t}(\mathbf{1}_{B})(x)\geq e^{2\lambda\psi(\lambda,t)(x)}-e^{-2\lambda r}.$$
\eqref{equ:G-ineq} allows us to estimate the first term in this lower bound.
If
$$\phi(\lambda C(\lambda,c),t)=\lambda^{2}C(\lambda,c)^{2}\int_{0}^{t} G\left(\frac{1}{\lambda^{2}C(\lambda,c)^{2}\tau}\right)d\tau,$$
\eqref{equ:G-ineq} yields
$$1=e^{2\lambda f(x)}\leq e^{2\lambda\psi(\lambda,t)(x)} e^{2\phi(\lambda C(\lambda,c),t)}.$$

Hence
$$P_{t}(\mathbf{1}_{B})(x)\geq e^{-2\phi(\lambda C(\lambda,c),t)}-e^{-2\lambda r}.$$

Choose $\lambda C(\lambda,c) =\frac{1}{r}$, $t = Ar^{2}$, and we obtain
$$P_{Ar^{2}}(\mathbf{1}_{B})(x)\geq e^{-2\phi(\frac{1}{r},Ar^{2})}-e^{-\frac{2}{C(\lambda,c)}}.$$

To finish, we must choose $A > 0$ sufficiently small, depending only on $n$, and a $\rho > 0$, so that for every $x\in V$ and $r>\frac{1}{2}$
\begin{equation}
e^{-2\phi(\frac{1}{r},Ar^{2})}-e^{-\frac{2}{C(\lambda,c)}} \geq \rho.
\label{equ:10}
\end{equation}
(Note that, actually, the point that $r > \frac{1}{2}$ simply implies that the term $e^{-\frac{2}{C(\lambda,c)}}$ is not one.  Replacing this by $r > \epsilon$ for any positive $\epsilon$ would likewise suffice.)

To see such an $A$ exists,  consider the function
$$\phi(\frac{1}{r},Ar^{2})=\frac{1}{r^{2}}\int_{0}^{Ar^{2}}G\left(\frac{r^{2}}{\tau}\right)d\tau=\int_{A^{-1}}^{\infty}\frac{G(t)}{t^{2}}dt.$$
$\phi(\frac{1}{r},Ar^{2}) \to 0$ as $A \to 0^+$, and hence such a sufficiently small $A$ exists to ensure that \ref{equ:10} holds and this completes the proof.
\end{proof}


In this section we use the previous result to show that non-negatively curved graphs (with respect to $CDE'$) satisfying the volume doubling property.  That is, we obtain:
\begin{theorem}\label{th:doubling-volume}
Suppose a locally finite, connected graph $G$ satisfies $CDE'(n,0)$, then $G$ satisfies the volume doubling property $DV(C)$. That is, there exists a constant $C=C(n) > 0$ such that for all $x\in V$ and all $r>0$:
\begin{equation*}
V(x,2r)\leq C  V(x,r).
\end{equation*}
\end{theorem}

Actually, with some simple computations we can get some slightly stronger conclusions on volume regularity that we will find useful in the proof of a Gaussian estimate.

\begin{remark}\label{rem:general-doubling-volume}
For any $r\geq s$,
\[
\begin{split}
V(x,r)
&\leq V(x,2^{[\frac{\log(\frac{r}{s})}{\log2}]+1}s)\\
&\leq  C^{1+\frac{\log(\frac{r}{s})}{\log2}}V(x,s)\\
&=C \left(\frac{r}{s}\right)^{\frac{\log C}{\log2}}V(x,s),
\end{split}
\]
where $[x]$ denotes the integer part of $x$.
\end{remark}

One final tool in the proof of Theorem~\ref{th:doubling-volume} is an explicit form of a Harnack inequality arising from the Li-Yau inequality.  Such an inequality was derived in \cite{BHLLMY13}.  In the (simplified by our assumption that $K=0$) form in which we apply it, it states that

\begin{corollary}\label{coro:hanack}
Suppose $G$ is a finite or infinite graph satisfying $CDE'(n,0)$, and assume $D:=\frac{\mu_{max}}{\omega_{min}} <\infty$, then for every $x\in V$ and $(t, y),(t, z)\in (0,+\infty)\times V$ with $t < s$ one has
$$p(t,x,y)\leq  p(s,x, z)\left(\frac{s}{t}\right)^{n} \exp\left(\frac{4Dd(y,z)^{2}}{s-t}\right).$$
\end{corollary}

We now turn to the proof of Theorem~\ref{th:doubling-volume}.
\begin{proof}
From the semigroup property and the symmetry of the heat kernel in Remark~\ref{rem:heatkpro}, for any $y\in V$ and $t>0$, we have
$$p(2t,y,y)=\sum_{z \in V}\mu(z)p(t,y,z)^{2}.$$
Consider now a cut-off function $h \in V^{\mathbb{R}}$ such that $0\leq h\leq 1$, $h\equiv 1$ on $B(x,\frac{\sqrt{t}}{2})$ and $h\equiv 0$ outside $B(x,\sqrt{t})$. We thus have
\begin{equation}
\label{equ:hkernel-lower-bound}
\begin{split}
P_{t}h(y)
&=\sum_{z \in V}\mu(z)p(t,y,z)h(z)\\
&\leq \left(\sum_{z \in V}\mu(z)p(t,y,z)^{2}\right)^{\frac{1}{2}}\left(\sum_{z \in V}\mu(z)h(z)^{2}\right)^{\frac{1}{2}}\\
&\leq\left(p(2t,y,y)\right)^{\frac{1}{2}}\left(V(x,\sqrt{t})\right)^{\frac{1}{2}}.
\end{split}
\end{equation}
Taking $y = x$, and $t = r^{2}$, we obtain
\begin{equation}
\left(P_{r^{2}}(\mathbf{1}_{B(x,\frac{r}{2})})(x)\right)^{2}\leq \left(P_{r^{2}}h(x)\right)^{2} \leq p(2r^{2},x,x)V(x,r).
\label{equ:11}
\end{equation}
At this point we use the crucial inequality \eqref{equ:exp-ineq}, which gives for some $0 < A < 1$, depending on the dimension $n$,
$$P_{Ar^{2}} \left(\mathbf{1}_{B(x,r)}\right)(x)\geq \rho, \quad x\in V,\quad r>\frac{1}{2}.$$

Combining the latter inequality with \eqref{equ:11} and Corollary~\ref{coro:hanack}, we obtain an on-diagonal lower bound
\begin{equation}
p(2r^{2},x,x)\geq\frac{\rho^*}{V(x,r)},\quad x\in V,\quad r>\frac{1}{2}.
\label{equ:12}
\end{equation}
Applying Corollary~\ref{coro:hanack} to $p(t,x, y)$, one obtains that for every $y\in B(x,\sqrt{t})$, we find
\begin{equation}
p(t,x, x)\leq C(n)p(2t,x, y).
\label{equ:13}
\end{equation}
Integrating the above inequality over $B(x,\sqrt{t})$ with respect to $y$ gives
$$p(t,x, x)V(x,\sqrt{t})\leq C(n)\sum_{y \in B(x,\sqrt{t})}\mu(y)p(2t,x, y)\leq C(n).$$
Further letting $t = 4r^{2}$, we obtain an on-diagonal upper bound
\begin{equation}
p(4r^{2},x, x)\leq \frac{C(n)}{V(x,2r)}.
\label{equ:14}
\end{equation}
Combining \eqref{equ:12},\eqref{equ:13} with \eqref{equ:14} we finally obtain for any $r>\frac{1}{2}$,
$$V(x,2r)\leq \frac{C}{p(4r^{2},x, x)}\leq \frac{C^{*}}{p(2r^{2},x, x)}\leq C^{**} V(x,r).$$
When $0<r\leq\frac{1}{2}$, $V(x,2r)\equiv  V(x,r)=\mu(x)$, the result is obvious. This completes the proof.
\end{proof}
As a remark, while the proof is fairly simple, it illustrates the power of inequality \eqref{equ:11}.  The failing in the \cite{BHLLMY13} paper to obtain volume doubling lay exactly in this point: we previously were obtained a lower bound on this quantity by directly applying the Harnack inequality Corollary \ref{coro:hanack} in a somewhat unusual fashion, rather strongly using the fact that the Harnack inequality obtained by integrating the Li-Yau inequality gives a more explicit estimate than the continuous-time Harnack inequality introduced in Section 2.   We did this, instead of using the approach of the current proof, the lack of quality cutoff functions in the graph case meant that our Li-Yau inequality (and hence the obtained Harnack inequality) was insufficient to derive a strong enough lower bound to imply volume doubling.  The lesson here should be taken that {\it using the heat-semigroup arguments as done above allows us to work around the lack of quality cutoff functions for graphs.}

\section{Gaussian estimate} \label{sec:GE}
In this section we focus on the normalized Laplacian: that is, we take our measure $\mu$ to be $\mu(x) = m(x)$.   We will prove a discrete-time Gaussian estimate on a infinite, connected and locally finite graph $G=(V,E)$.

Let $\mathcal{P}_{t}(x,y)=p(t,x,y)m(y)$ be the continuous-time Markov kernel on the graph.  It is also a solution of the heat equation. By symmetry, the  heat kernel $p(t,x,y)$ satisfies
\[\frac{\mathcal{P}_{t}(x,y)}{m(y)}=\frac{\mathcal{P}_{t}(y,x)}{m(x)}.\]
Let $p_{n}(x,y)$ be the discrete-time kernel on $G$, which is defined by
\[\left\{
  \begin{array}{ll}
    p_{0}(x,y)=\delta_{xy},  \\
    p_{k+1}(x,z)=\sum_{y\in V}p(x,y)p_{k}(y,z).
  \end{array}
\right.\]
where $p(x,y):=\frac{\omega_{xy}}{m(x)}$, and $\delta_{xy}=1$ only when $x=y$, otherwise equals to $0$. We can know the two kernels satisfy
\begin{equation}\label{equ:cont-disc-kernel}
e^{-t}\sum_{k=0}^{+\infty}\frac{t^{k}}{k!}p_{k}(x,y)=\mathcal{P}_{t}(x,y).
\end{equation}

In order to obtain our desired Gaussian estimate, we first establish the continuous-time Gaussian on-diagonal estimate for graphs. As demonstrated in \cite{BHLLMY13}, the Harnack inequality obtained in that paper suffices to prove a Gaussian {\it upper} bound for bounded degree graphs satisfying $CDE(n,0)$.  A Gaussian lower bound was unavailable in that work, however. This failure was closely tied to the inability to use $CDE'$ to imply volume doubling.  With the new information gleaned from our modified curvature condition, we {\it are} however able to derive the Gaussian lower bound, as we now illustrate.

\begin{theorem}\label{th:continue-time-Gauss-estimate}
Suppose a graph $G$ satisfies $CDE'(n_0,0)$.  Then $G$ satisfies the continuous-time Gaussian estimate.   That is,  there exists a constant $C$ with respect to $n_0$ so that, for any $x,y\in V$ and for all $t>0$,
\[\mathcal{P}_{t}(x,y)\leq \frac{Cm(y)}{V(x,\sqrt{t})}.\]
Furthermore, for any $t_0 > 0$, there exist constants $C'$ and $c'$, so that for all $t > t_0$:
\[\mathcal{P}_{t}(x,y)\geq \frac{C'm(y)}{V(x,\sqrt{t})}\exp\left(-c'\frac{d(x,y)^{2}}{t}\right).\]
\end{theorem}
\begin{proof}
The upper bound follows from the methods of \cite{BHLLMY13}, as the Harnack inequality obtained in that paper still applies for graphs satisfying $CDE'(n_0,0)$.  For completeness, we include the brief proof. From Corollary~\ref{coro:hanack}, for any $t>0$, choosing $s=2t$ and for any $z\in B(x,\sqrt{t})$, we have
\[p(t,x,y)
\leq p(2t,z,y)2^{n_0}\exp(4D).\]
Integrating the above inequality over $B(x,\sqrt{t})$ with respect to $z$, gives
\[
\begin{split}
p(t,x,y)
&\leq \frac{C}{V(x,\sqrt{t})}\sum_{z\in B(x,\sqrt{t})}\mu(z)p(2t,z,y)\\
&\leq \frac{C}{V(x,\sqrt{t})}.
\end{split}
\]

We now prove the lower bound estimate. Recall that we only claim the result under the assumption that $t > t_0$.  The result is most transparent if $t_0 > 1/2$.  In this case,  then taking $t > 1/2$ and choosing $2r^{2}=\varepsilon t$ for some $0<\varepsilon<1$, equation \eqref{equ:12} implies that every $x \in V$ satisfies
\begin{equation}
p(\varepsilon t,x,x)\geq \frac{\rho^*}{V(x,\sqrt{\frac{\varepsilon t}{2}})} \geq \frac{\rho^*}{V(x,\sqrt{t})}.
\label{equ:15}
\end{equation}
Applying Corollary~\ref{coro:hanack}, taking $\varepsilon t$ as `$t$', taking $t$ to be `$s$', and choosing $y=x, z=y$, we obtain
\begin{equation}p(\varepsilon t,x,x)\leq p(t,x,y)\varepsilon^{n_0}\exp\left(\frac{4Dd(x,y)^{2}}{(1-\varepsilon)t}\right).
\label{equ:16}
\end{equation}
Combining \eqref{equ:15} with \eqref{equ:16}, we finally obtain for any $t>\frac{1}{2}$,
\[p(t,x,y)\geq \frac{\varepsilon^{-n_0}\rho^*}{V(x,\sqrt{t})}\exp\left(-\frac{4Dd(x,y)^{2}}{(1-\varepsilon)t}\right)
=\frac{C'}{V(x,\sqrt{t})}\exp\left(-c'\frac{d(x,y)^{2}}{t}\right).\]
While we assumed that $t > \frac{1}{2}$ here, if we fix {\it any} $t_0 > 0$, it is easy to rework the proof Theorem~\ref{th:exp-ineq} to work with such an arbitrary $t_0$ and this completes the proof of the theorem.
\end{proof}

The remaining difficulty is verifiying that the lower bound holds when $t$ is small enough.  This we will defer until after proving discrete-time Gaussian estimate to Remark~\ref{cont-time-Gauss-lowb2}. Together, this will complete the proof of Theorem~\ref{th:continue-time-Gauss-estimate}.

As a special case, note that if $t\geq \max\left\{d(x,y)^{2},{\frac{1}{2}}\right\}$, then the lower estimate can be write
\begin{equation}\label{equ:hkernel-lower-bound2}
p(t,x,y)\geq \frac{C''}{V(x,\sqrt{t})}.
\end{equation}
Before we ultimately finish the continuous-time lower bound, we address the discrete-time estimate.  We begin with the on-diagonal estimate:
\begin{proposition}\label{pro:dis-time-on-diag}
Assume a graph $G$ satisfies $CDE'(n_{0},0)$ and $\Delta(\beta)$, then there exist $c_{d},C_{d}>0$, for any $x,y \in V$, such that,
\[p_{n}(x,y)\leq \frac{C_{d}m(y)}{V(x,\sqrt{n})},\quad \mbox{for~all}~n>0,\]
\[p_{n}(x,y)\geq \frac{c_{d}m(y)}{V(x,\sqrt{n})},\quad \mbox{if}~n\geq d(x,y)^{2}.\]
\end{proposition}
This proposition follows the methods of Delmotte from \cite{D99}. To prove it, we first introduce some necessary results. Assume $\Delta(\alpha)$ is true (cf. Definition 2.1), so that we can consider the positive submarkovian kernel
\[\overline{p}(x,y)=p(x,y)-\alpha \delta_{xy}.\]
Now, compute $\mathcal{P}_{n}(x,y)$ and $p_{n}(x,y)$ with $\overline{p}(x,y)$,
\[\mathcal{P}_{n}(x,y)=e^{(\alpha-1)n}\sum_{k=0}^{+\infty}\frac{n^{k}}{k!}\overline{p}_k(x,y)=\sum_{k=0}^{+\infty}a_{k}\overline{p}_k(x,y),\]
\[p_{n}(x,y)=\sum_{k=0}^{n}C_{n}^{k}\alpha^{n-k}\overline{p}_k(x,y)=\sum_{k=0}^{n}b_{k}\overline{p}_k(x,y).\]
There is a lemma from \cite{D99} to compare the two sums as follows.
\begin{lemma}\label{lem:coeff-comp}
Let $c_{k}=\frac{b_{k}}{a_{k}}$, for $0\leq k \leq n$, and suppose $\alpha \leq \frac{1}{4}.$
\begin{itemize}
  \item  $c_{k}\leq C(\alpha)$, when $0\leq k\leq n$,
  \item  $c_{k}\geq C(a,\alpha)>0$, when $n\geq \frac{a^{2}}{\alpha^{2}}$ and $|k-(1-\alpha)n|\leq a\sqrt{n}$.
\end{itemize}
\end{lemma}
Note that the condition that $\alpha \leq \frac{1}{4}$ implies that $\frac{n}{2} \leq k \leq n$ in the second assertion.  Note that assuming $\alpha \leq \frac{1}{4}$ does not inhibit us: it is clear from the definition that if $\Delta(\alpha)$ holds, so does $\Delta(\alpha')$ for any $\alpha' < \alpha$.

Now we turn to the proof of Proposition~\ref{pro:dis-time-on-diag}.
\begin{proof}[Proof of Proposition \ref{pro:dis-time-on-diag} ]
The proof comes from Delmotte of \cite{D99}.

The first assertion in Lemma~\ref{lem:coeff-comp} implies, for any $n\in \mathbb{N}$,
\[p_{n}(x,y)\leq C(\beta)\mathcal{P}_{n}(x,y).\]
The upper bound, then, is an immediate consequence of Theorem~\ref{th:continue-time-Gauss-estimate}:  For any $x,y\in V$,
\[p_{n}(x,y)\leq \frac{C(\beta)Cm(y)}{V(x,\sqrt{n})}=\frac{C_{d}m(y)}{V(x,\sqrt{n})}.\]
The second assertion is a little more complicated.

Suppose, for a minute, that for any $\varepsilon > 0$, there exists an $a > 0$ such that
\begin{equation}
\sum_{|k-(1-\alpha)n|>a\sqrt{n}}a_{k}\overline{p}_k(x,y)\leq\frac{\varepsilon m(y)}{V(x,\sqrt{n})}.
\label{equ:17}
\end{equation}
We return briefly to prove that such an $a$ always exists.

Fix such an $a$ for a sufficiently small $\varepsilon$, taking, say, \[0 < \varepsilon < \frac{1}{2} C' \leq \frac{1}{2} \mathcal{P}_n(x,y) \cdot \frac{V(x,\sqrt{n})}{m(y)}.\]  We set $\alpha=\frac{\beta}{2}$, and $n\geq N=\frac{a^{2}}{\alpha^{2}}$.

For such choices we have, by the second assertion of Lemma \ref{lem:coeff-comp}, that
\begin{align*}
 p_{n}(x,y)
&\geq\sum_{|k-(1-\alpha)n|\leq a\sqrt{n}}b_{k}\overline{p}_k(x,y)\\
&\geq C(a,\alpha)\sum_{|k-(1-\alpha)n|\leq a\sqrt{n}}a_{k}\overline{p}_k(x,y),
\end{align*}
and furthermore
\[
\begin{split}
&C(a,\alpha)\mathcal{P}_{n}(x,y)\\
=&C(a,\alpha)\sum_{|k-(1-\alpha)n|\leq a\sqrt{n}}a_{k}\overline{p}_k(x,y)+C(a,\alpha)\sum_{|k-(1-\alpha)n|> a\sqrt{n}}a_{k}\overline{p}_k(x,y)\\
\leq& p_{n}(x,y)+C(a,\alpha)\sum_{|k-(1-\alpha)n|>a\sqrt{n}}a_{k}\overline{p}_k(x,y)\\
\leq& p_{n}(x,y)+C(a,\alpha)\frac{\varepsilon m(y)}{V(x,\sqrt{n})}.
\end{split}
\]
Since we assume $n\geq d(x,y)^{2}$, applying the second assertion of \eqref{equ:hkernel-lower-bound2}, we obtain
\[
\begin{split}
p_{n}(x,y)&\geq C(a,\alpha)\left(\mathcal{P}_{n}(x,y)-\frac{\varepsilon m(y)}{V(x,\sqrt{n})}\right)\\
&\geq C(a,\alpha)\left(\frac{C' m(y)}{V(x,\sqrt{n})}-\frac{\varepsilon m(y)}{V(x,\sqrt{n})}\right)\\
&= \frac{c_{d}m(y)}{V(x,\sqrt{n})},
\end{split}
\]
as we desired.

Thus it remains to prove that \eqref{equ:17} can be satisfied. First we consider another, slightly modified, Markov kernel
$p'=\frac{\overline{p}}{1-\alpha}$.  Such a kernel is generated by weights $\omega'_{xy}$ as follows:
\begin{align*}
\omega'_{xx} &=\frac{\omega_{xx}-\alpha m(x)}{1-\alpha}\geq \alpha m(x),\quad\forall x\in V,\\
\omega'_{xy}&=\frac{\omega_{xy}}{1-\alpha},\quad\forall x\neq y \in V,\\
m'(x)&=m(x).
\end{align*}
Note that $\Delta(\alpha)$ is true in $G$ with the new weights.  The condition $CDE'(n_{0},0)$ is also still holds for the new weights, because if one lets $\Delta'$ be the new Laplacian for $\omega'_{xy}$, then for any $f,g \in V^{\mathbb{R}}$ we obtain:
\[\Delta'f(x)=\frac{1}{1-\alpha}\Delta f(x), \quad \Gamma'(f,g)=\frac{1}{1-\alpha}\Gamma(f,g),\]
\[\Gamma_{2}'(f,g)=\frac{1}{(1-\alpha)^{2}}\Gamma_{2}(f,g), \quad \widetilde{\Gamma}_{2}'(f)=\frac{1}{(1-\alpha)^{2}}\widetilde{\Gamma}_{2}(f).\]
Furthermore, the process of proving $DV(C)$ still works when adding loops to every point of graph. Then $DV(C)$ is still satisfied for the new weights. This yields
\[p_{k}'(x,y)\leq \frac{C'_{d}m(y)}{V(x,\sqrt{k})},\]
and hence
\[\overline{p}_{k}(x,y)\leq \frac{C_{d}'m(y)(1-\alpha)^{k}}{V(x,\sqrt{k})}.\]
Next, we have to get the estimate as follows
\[e^{(\alpha-1)n}\sum_{|k-(1-\alpha)n|>a\sqrt{n}} \frac{((1-\alpha)n)^{k}}{k!}\frac{1}{V(x,\sqrt{k})}\leq\frac{\varepsilon' }{V(x,\sqrt{n})}.\]
The sum for $k>a\sqrt{n}+(1-\alpha)n$ is easier because we simply use
\[V(x,\sqrt{k})\geq V\left(x,\sqrt{\frac{n}{2}}\right)\geq V\left(x,\frac{\sqrt{n}}{2}\right)\geq \frac{V(x,\sqrt{n})}{C_{1}},\]
so we have,
\[
\begin{split}
&e^{(\alpha-1)n}\sum_{k>a\sqrt{n}+(1-\alpha)n} \frac{((1-\alpha)n)^{k}}{k!}\frac{1}{V(x,\sqrt{k})}\\
&\leq e^{(\alpha-1)n}\frac{C_{1}}{V(x,\sqrt{n})}\sum_{k>a\sqrt{n}+(1-\alpha)n} \frac{((1-\alpha)n)^{k}}{k!}\\
&\leq e^{(\alpha-1)n}\frac{C_{1}}{V(x,\sqrt{n})} \frac{((1-\alpha)n)^{(1-\alpha)n+a\sqrt{n}}}{(a\sqrt{n}
+(1-\alpha)n)!}\frac{1}{1-\frac{(1-\alpha)n}{a\sqrt{n}+(1-\alpha)n}}\\
&\leq \frac{CC_{1}}{V(x,\sqrt{n})}\exp\left(a\sqrt{n}-(a\sqrt{n}+(1-\alpha)n)
\log\left(1+\frac{a}{(1-\alpha)\sqrt{n}}\right)\right)\\
&\cdot \frac{1}{\sqrt{a\sqrt{n}+(1-\alpha)n}}
\frac{a\sqrt{n}+(1-\alpha)n}{a\sqrt{n}}\\
&\leq \frac{\varepsilon' }{2 V(x,\sqrt{n})},
\end{split}
\]
for our (arbitrary) choice of $\varepsilon'$, so long as $a$ is sufficiently large.  Note that the second to last inequality follows from the fact that $k!\geq \frac{ k^ke^{-k}\sqrt{k}}{C}$.  To see that this last line holds, note that as we assume $n \geq \frac{a^2}{\alpha^2}$, the inequality  $\frac{1}{\sqrt{a\sqrt{n}+(1-\alpha)n}}
\frac{a\sqrt{n}+(1-\alpha)n}{a\sqrt{n}}\leq \frac{1}{a}$ holds.  Finally observe that, by the real number inequality $\log(1+u) \geq \frac{u}{1+u} +  \frac{u^2}{2(1+u)^2}$, then the exponential is negative and for a sufficiently large $a$ the claim holds.

Remark~\ref{rem:general-doubling-volume} allows us to deal with $1\leq k<-a\sqrt{n}+(1-\alpha)n$.  It gives
\[V(x,\sqrt{k})\leq C\left(\frac{\sqrt{k}}{\sqrt{k-1}}\right)^{\frac{\log C}{\log 2}}V(x,\sqrt{k-1})\leq C_{2}V(x,\sqrt{k-1}).\]
Thus, for the terms $1\leq k\leq \frac{(1-a)n}{2C_{2}}$, we have
\[
\frac{((1-\alpha)n)^{k-1}}{(k-1)!}\frac{1}{V(x,\sqrt{k-1})}\leq \frac{1}{2}\frac{((1-\alpha)n)^{k}}{k!}\frac{1}{V(x,\sqrt{k})},
\]
and the estimate is straightforward. For the other term $ \frac{(1-a)n}{2C_{2}}< k<-a\sqrt{n}+(1-\alpha)n$, and using Remark~\ref{rem:general-doubling-volume} again
\[V(x,\sqrt{k})\geq V\left(x,\sqrt{\frac{(1-a)n}{2C_{2}}}\right)\geq C_{3}V(x,\sqrt{n}).\]
This completes the proof.
\end{proof}

In proving the upper bounds of the discrete-time Gaussian estimate on graphs, it is useful to introduce the following result from \cite{CG98}.
\begin{theorem}\label{th:equ-pro}
For a reversible nearest neighborhood random walk on the locally
finite graph $G=(V,E)$, the following properties are equivalent:
\begin{enumerate}
  \item The relative Faber-Krahn inequality $(FK)$.
  \item The discrete-time Gaussian upper estimate in conjunction with the doubling property $DV(C)$.
  \item The discrete-time on-diagonal upper estimate in conjunction with the doubling property $DV(C)$.
\end{enumerate}
\end{theorem}

Now we show the final theorem of the discrete-time Gaussian estimate.
\begin{theorem}\label{th:dis-time-Gauss-estimate}
Assume a graph $G$ satisfies $CDE'(n_{0},0)$ and $\Delta(\alpha)$, then the graph satisfies the discrete-time Gaussian estimate $G(c_{l},C_{l},C_{r},c_{r})$.
\end{theorem}
\begin{proof}
We have already observed that the discrete-time on-diagonal upper estimate and the doubling property $DV(C)$  both hold for graphs satisfying $CDE'(n_{0},0)$ and $\Delta(\alpha)$.  Theorem~\ref{th:equ-pro}  immediately implies the discrete-time Gaussian upper estimate.

The lower bound follows from the on-diagonal one. The strategy
is similar to Delmotte of \cite{D99}. We repeatedly apply the second assertion of Proposition~\ref{pro:dis-time-on-diag}. Set $n=n_{1}+n_{2}+\cdots+n_{j}$, $x=x_{0},x_{1},\cdots,x_{j}=y$ and $B_{0}={x}$, $B_{i}=B(x_{i},r_{i})$, $B_{j}={y}$, such that
\[
\left\{
  \begin{array}{ll}
    j-1\leq C\frac{d(x,y)^{2}}{n},  \\
    r_{i}\geq c\sqrt{n_{i+1}}, & \hbox{so that $V(z,\sqrt{n_{i+1}})\leq A V(B_{i})$, when $z\in B_{i}$,} \\
  \sup_{z\in B_{i-1},z'\in B_{i}}d(z,z')^{2}\leq n_{i},& \hbox{so that $p_{n_{i}}(z,z')\geq \frac{c_{d}m(z')}{V(z,\sqrt{n_{i}})}$.}
  \end{array}
\right.
\]
Such a decomposition allows us to immediately derive the lower bound.  Indeed,
\[
\begin{split}
p_{n}(x,y)
&\geq \sum_{(z_{1},\cdots,z_{j-1})\in B_{1}\times\cdots\times B_{j-1}} p_{n_{1}}(x,z_{1})p_{n_{2}}(z_{1},z_{2})\cdots p_{n_{j}}(z_{j-1},y)\\
&\geq \sum_{(z_{1},\cdots,z_{j-1})\in B_{1}\times\cdots\times B_{j-1}} \frac{c_{d}m(z_{1})}{V(x,\sqrt{n_{1}})} \frac{c_{d}m(z_{2})}{V(z_{1},\sqrt{n_{2}})}\cdots \frac{c_{d}m(y)}{V(z_{j-1},\sqrt{n_{j}})}\\
&\geq c_{d}^{j}A^{1-j}\sum_{(z_{1},\cdots,z_{j-1})\in B_{1}\times\cdots\times B_{j-1}}\frac{m(z_{1})}{V(x,\sqrt{n_{1}})} \frac{m(z_{2})}{V(B_{1})}\cdots \frac{m(y)}{V(B_{j})}\\
&=\frac{c_{d}m(y)}{V(x,\sqrt{n_{1}})}\left(\frac{c_{d}}{A}\right)^{j-1}.
\end{split}
\]

If we choose $C_{l}\geq C\log(\frac{A}{c_{d}})$, and $V(x,\sqrt{n_{1}})\leq V(x,\sqrt{n})$,  the Gaussian lower bound
\[p_{n}(x,y)\geq \frac{c_{d}m(y)}{V(x,\sqrt{n})} e^{-C_{l}\frac{d(x,y)^{2}}{n}},\]
and thus the theorem, follows.
\end{proof}
Now from the discrete-time Gaussian estimate, we can get the continuous-time Gaussian lower bound estimate when $t$ is small enough we have mentioned before in Theorem~\ref{th:continue-time-Gauss-estimate}, as follows.
\begin{remark}\label{cont-time-Gauss-lowb2}
If $t$ is small enough, under $CDE'(n_0,0)$, for any $x,y\in V$,
\[\mathcal{P}_{t}(x,y)\geq \frac{C'm(y)}{V(x,\sqrt{t})}\exp\left(-c'\frac{d(x,y)^{2}}{t}\right)\]
holds too.
\end{remark}
\begin{proof}
If $d(x,y)=0$, $\mathcal{P}_{t}(x,y)=m(y)p(t,x,y)\rightarrow1$, when $t\rightarrow0^+$. It naturally satisfies the lower bound.

Now we consider $d(x,y)>0$. When $k<d(x,y)$, then $p_k(x,y)=0$. When $k\geq d(x,y)$, form the discrete-time Gaussian lower bound, the relationship between the continuous-time heat kernel and the discrete-time heat kernel \eqref{equ:cont-disc-kernel}, and the polynomial volume growth of Corollary 7.8 in \cite{BHLLMY13}, which says under $CDE(n_0,0)$, there exists a constant $C_0>0$, such that
\[V(x,\sqrt{t})\leq C_0\mu(x)t^{n_0},\]
for $CDE'(n_0,0)\Rightarrow CDE(n_0,0)$, and $\mu(x)=m(x)$ in this section, we obtain
\[\begin{split}
\mathcal{P}_{t}(x,y)=e^{-t}\sum_{k=0}^{+\infty}\frac{t^{k}}{k!}p_{k}(x,y)
&\geq e^{-t}\sum_{k=d(x,y)}^{+\infty}\frac{t^{k}}{k!}\frac{c_{d}m(y)}{V(x,\sqrt{k})} e^{-C_{l}\frac{d(x,y)^{2}}{k}}\\
&\geq e^{-t}\sum_{k=d(x,y)}^{+\infty}\frac{t^{k}}{k!}\frac{c_{d}m(y)}{C_0k^{n_0}m(x)} e^{-C_{l}d(x,y)}\\
&\geq\frac{c_{d}e^{-t}}{C_0}\frac{m(y)}{m(x)}\cdot\frac{e^{-C_{l}d(x,y)}t^{d(x,y)}}{d(x,y)^{n_0}d(x,y)!}\\
&\geq\frac{C'm(y)}{V(x,\sqrt{t})}\exp\left(-c'\frac{d(x,y)^{2}}{t}\right),~~~~(t\rightarrow 0^+). \\
\end{split}\]
In the last step, when $t$ is small enough, $V(x,\sqrt{t})=m(x)$, and $e^{-t}$ is bounded. Furthermore, for any $x,y\in V$, $m=d(x,y)$ is finite in connected graphs, and it mainly dues to if $t\rightarrow 0^+$, the following function for all $m\in \mathbb{Z}_+$,
$$f(t,m)=-\frac{t}{m^2}\ln\left(\frac{e^{-C_{l}m}t^{m}}{m^{n_0}m!}\right)$$
has positive bounds. Since for all $m\in \mathbb{Z}_+$,
\[f(t,m)=-\frac{t\ln t}{m}+t\left(\frac{C_{l}}{m}+n_0\frac{\ln m}{m^2}+\frac{\ln m!}{m^2}\right)\rightarrow 0,\]
when $t\rightarrow 0^+$. So we can always find a $c'>0$ independent of $d(x,y)$ and the above inequality holds. This completes the proof.
\end{proof}
Here we have the following result.
\begin{theorem}\label{maintheorem1}
If the graph satisfies $CDE'(n_{0},0)$ and $\Delta(\alpha)$, we have the following four properties.
\begin{enumerate}
  \item[$1)$]{} There exists $C_1, C_2, \alpha>0$ such that $DV(C_1)$, $P(C_2)$, and $\Delta(\alpha)$ are true.
  \item[$2)$]{} There exists $c_{l},C_{l},C_{r},c_{r}>0$ such that $G(c_{l},C_{l},C_{r},c_{r})$ is true.
  \item[$3)$]{} There exists $C_H$ such that $H(\eta,\theta_1,\theta_2,\theta_3,\theta_4,C_H)$ is true.
  \item[$3)'$]{} There exists $C_\mathcal{H}$ such that $\mathcal{H}(\eta,\theta_1,\theta_2,\theta_3,\theta_4,C_\mathcal{H})$ is true.
\end{enumerate}
\end{theorem}
\begin{proof}
The condition $CDE'(n_{0},0)$ implies $DV(C_1)$ (see Theorem~\ref{th:doubling-volume}), and Theorem~\ref{th:dis-time-Gauss-estimate} states that $CDE'(n_{0},0)$ and $\Delta(\alpha)$ implies $G(c_{l},C_{l},C_{r},c_{r})$. According to Delmotte of \cite{D99}, $P(C_2)$ is true. Moreover, $3)$ and $3)'$ hold too.
\end{proof}

\section{Diameter bounds} \label{sec:DB}
In this section we will show another application of Theorem \ref{th:main-var-ineq}.  We prove that positively curved graphs (that is graphs satisfying $CDE'(n,K)$ for some $K > 0$ are finite.  In order to prove this, we
consider an alternate to the natural distance function on a graph, the so-called canonical distance and diameter of $G$ associated with a Laplace operator $\Delta:$
\[\widetilde{d}(x,y)=\sup_{f\in \ell^{\infty}(V,\mu),\|\Gamma(f)\|_{\infty}\leq1}|f(x)-f(y)|,~~x,y\in V.\]
\[\widetilde{D}=\sup_{x,y\in V} \widetilde{d}(x,y).\]
In this section we are concerned with simple, connected and loopless graphs.
\subsection{Global heat kernel bounds}

In this subsection we derive a global heat kernel bound by proving finite measure under the assumption of positive curvature on graphs, and use this to establish that the diameter is finite.  We accomplish this in several steps the most crucial of them being an estimate proving that the total measure of the graph is finite.

In Theorem~\ref{th:main-var-ineq}, we choose the function $\gamma$ in a such a way that,
$$\alpha'-\frac{4\alpha\gamma}{n}+2\alpha K=0,$$
that is
$$\gamma=\frac{n}{4}\left(\frac{\alpha'}{\alpha}+2K\right).$$
Integrating both sides of the inequality~\eqref{equ:main-var-ine} from $0$ to $T$, we obtain
\begin{equation}
\alpha(T)\frac{P_T(\Gamma(\sqrt{f}))}{P_Tf}-\alpha(0)\frac{\Gamma(\sqrt{P_Tf})}{P_Tf}\geq \frac{2}{n}\left(\int_0^T \alpha\gamma dt\right) \frac{\Delta P_T(f)}{P_Tf}-\frac{2}{n}\int_0^T \alpha\gamma^2 dt. \label{equ:18}
\end{equation}
Now we introduce the first result in this subsection.

\begin{proposition}\label{pro:difference-ineq-hkernel}
Let $G=(V,E)$ be a locally finite, connected graph satisfying $CDE'(n,K)$. Then for all $0< \theta <K$ and $t_0>0$, there exists a $C_1>0$ such that for every non-negative $f$ satisfying $||f||_{\infty} \leq 1$, and every $t \geq t_0$,
\[|\sqrt{P_tf}(x)-\sqrt{P_tf}(y)|\leq C_1e^{-\frac{\theta}{2} t}\widetilde{d}(x,y),~~x,y\in V.\]
\end{proposition}
{\bf Remark:}  Of course it is easy to replace the assumption that $||f||_{\infty} \leq 1$ by $||f||_\infty \leq M$ for any $M \geq 0$.

\begin{proof}
Fix some $0 < \theta < K$ and some $0 < t_0 \leq T$.  We show the inequality holds at time $T$ assuming $T$ is sufficiently large. In \eqref{equ:18},  we take
\[\alpha(t)= 2K e^{- \theta t}(e^{- \theta t}-e^{- \theta T})^{2K/\theta-1},\]
so that
\[\alpha(0)=2K(1-e^{-\theta T})^{2K/\theta-1},~~\mbox{and}~\alpha(T)=0.\]
With such choice a simple computation gives,
\[\gamma=\frac{n}{4}\left(-e^{-\theta T}\frac{2K-\theta}{e^{-\theta t}-e^{-\theta T}}\right),\]
which is non-positive for $0 \leq t \leq T$.

Then, for any $T>0$,
\begin{equation}\label{equ:19}
- Kn(1-e^{-\theta T})^{2K/\theta-1}\frac{\Gamma(\sqrt{P_Tf})}{P_Tf}\geq\left(\int_0^T\alpha\gamma dt\right)\frac{\Delta P_T(f)}{P_Tf}-\int_0^T \alpha \gamma^2 dt.
\end{equation}
Now, we can compute
\[\int_0^T \alpha \gamma dt= -\frac{nK}{2}(1-e^{-\theta T})^{2K/\theta-1}(e^{-\theta T}),\]
and
\[\int_0^T a\gamma^2 dt=
\frac{Kn^2 }{8}(1-e^{-\theta T})^{2K/\theta-2}e^{-2\theta T}
 \times \left( \frac{\theta(2K/\theta-1)^2}{2K/\theta-2}\right).
 \]
We thus obtain from \eqref{equ:19}, that for any $T>t_0 \geq 0$,
\begin{multline} \label{equ:lowb}
0\geq -Kn(1-e^{-\theta T})^{2K/\theta-1}\frac{\Gamma(\sqrt{P_Tf})}{P_Tf}
\\ \geq -\frac{nK}{2}(1-e^{-\theta T})^{2K/\theta-1}e^{-\theta T} \frac{\Delta P_Tf}{P_Tf}
-\frac{Kn^2 \theta(2K/\theta-1)^2}{8(2K/\theta-2)}(1-e^{-\theta T})^{2K/\theta-2}e^{-2\theta T}.
\end{multline}
Dividing, and switching notation from $T$ to $t$, we obtain that
\begin{equation}\label{equ:21}\begin{split}
\Gamma(\sqrt{P_tf})
&\leq \frac{1}{2}e^{-\theta t} \Delta P_tf+\frac{n\theta(2K/\theta-1)^2}{4(2K/\theta-2)(1-e^{-\theta t})}e^{-2 \theta t}P_tf\\
&\leq C_1^2e^{-\theta t},
\end{split}\end{equation}
with $C_1=\sqrt{D_\mu
+\frac{n\theta(2K/\theta-1)^2}{8(2K/\theta-2)(e^{\theta t_0}-1)}}$.

We consider the function $u(x)=\frac{1}{C_1}e^{\frac{\theta}{2} t}\sqrt{P_tf}(x)\in \ell^\infty(V,\mu)$.   By construction, we have normalized $u$ so that for any $t \geq t_0$,  $\|\Gamma(u)\|_\infty\leq 1$.    By the definition of the canonical distance $\widetilde{d}(x,y)$,
\[|u(x)-u(y)|\leq \widetilde{d}(x,y),\]
In turn,
\[|\sqrt{P_tf}(x)-\sqrt{P_tf}(y)|\leq C_1e^{-\frac{\theta}{2} t}\widetilde{d}(x,y).\]
as desired.
\end{proof}
\begin{proposition}\label{pro:partial-ineq-hkernel}
Let $G=(V,E)$ be a locally finite, connected graph satisfying $CDE'(n,K)$. Then for all $0<\theta<K$ and  $t_0>0$, there exists a $C_2>0$ such that for every non-negative function $f$ with $||f||_\infty \leq 1$, and for every $t\geq t_0$,
\begin{equation}\label{eq:pid}
|\partial_t P_tf|\leq C_2e^{-\frac{\theta}{2} t}.
\end{equation}
\end{proposition}
\begin{proof}
Let $P_tf=u$, we have
\begin{equation*}\begin{split}
|\Delta u|
&\leq\widetilde{\sum_{y\sim x}}|u(y)-u(x)|\\
&=\widetilde{\sum_{y\sim x}}\left(\sqrt{u}(y)+\sqrt{u}(x)\right)|\sqrt{u}(y)-\sqrt{u}(x)|\\
&\leq\left(\widetilde{\sum_{y\sim x}}\left(\sqrt{u(y)}+\sqrt{u(x)}\right)^2\right)^\frac{1}{2}\left(\widetilde{\sum_{y\sim x}}\left(\sqrt{u(y)}-\sqrt{u(x)}\right)^2\right)^\frac{1}{2}\\
&\leq2\sqrt{2D_\mu}C\sqrt{\Gamma(\sqrt{u})}.
\end{split}\end{equation*}
Combing with \eqref{equ:21}, we let $C_2=2\sqrt{2D_\mu} \cdot C_1$.  This yields the desired result.
\end{proof}
\begin{proposition}\label{pro:finite-measure}
Let $G=(V,E)$ be a locally finite, connected graph satisfying $CDE'(n,K)$ with $K>0$.  Then the measure $\mu$ is finite, that is, $\mu(V)<\infty$.
\end{proposition}
\begin{proof}
According to Proposition~\ref{pro:partial-ineq-hkernel} the limit of $p(t,x,\cdot)$ exists and is finite when $t\rightarrow\infty$.  Moreover Proposition~\ref{pro:difference-ineq-hkernel} and the property 2 in  Remark~\ref{rem:heatkpro}, imply that $\lim_{t\rightarrow\infty} p(t,x,\cdot)$ is some non-negative $c(x)\geq 0$.  The symmetry of the heat kernel implies that $c(x)$ actually does not depend on $x$.

To show the finiteness of the measure, it will suffice to prove that this limit is actually strictly positive under the assumption of $CDE'(n,K)$ for some $K>0$.

We apply lower bound in Proposition~\ref{pro:partial-ineq-hkernel}, integrating it from some $t_1>t_0$ to $\infty$ to obtain
\[\lim_{t\rightarrow\infty} p(t,x,\cdot)-p(t_1,x,y)=\int_{t_1}^\infty\partial_tp(t,x,\cdot)dt\geq -\frac{2 C_2}{s}e^{-\frac{\theta}{2} t_1},\]

Let $y=x$. Theorem 7 of \cite{BHLLMY13} states that there is some constant $C'>0$, so that
\[p(t_1,x,x)\geq\frac{C'}{t_1^n},\]
under the condition $CDE(n,0)$, and hence $CDE(n,K)$ -- and hence $CDE'(n,K)$ -- for any $K>0$.  Thus combining:
\[\lim_{t\rightarrow\infty} p(t,x,x)\geq \frac{C'}{t_1^n}-\frac{2 C_2}{\theta}e^{-\frac{\theta}{2} t_1}>0.\]
This implies that $\lim_{t\rightarrow\infty} p(t,x,y)=c>0$, for any $x,y\in V$. This, in turn, (from (3) in  Remark~\ref{rem:heatkpro}) implies that the measure $\mu$ is finite.
\end{proof}

Finally we introduce the following result (see \cite{GH14}) which says the properties of infinite measure and infinite diameter are equivalent properties for a locally compact separable metric space $M$, so long as the volume doubling ($DV$) holds.
\begin{lemma}\label{coro-RDV}
Assume that $(M, d)$ is connected and satisfies $DV$. Then
\[\mu(M)=\infty\Leftrightarrow diam(M) =\infty.\]
\end{lemma}
On graphs, we have the same assertion, and let the above $d$ is the natural distance on graphs. Since Theorem~\ref{th:doubling-volume}, we have already got $DV$ under the assumption $CDE'(n,0)$, and $CDE'(n,K)\Rightarrow CDE'(n,0)$, for any $K>0$. Combining with Proposition~\ref{pro:finite-measure} and the first equivalence in Lemma~\ref{coro-RDV}, we get the finiteness of diameter as follows.

\begin{theorem}\label{finite-diameter}
Every locally finite, connected, simple graph satisfying $CDE'(n,K)$ with $K>0$ has finite diameter in terms of the natural graph distance.
\end{theorem}

Per Proposition~\ref{pro:finite-measure}, we may assume $\mu$ is probability measure -- renormalizing so that $\lim_{t\rightarrow\infty} p(t,x,\cdot)=1$.

\begin{proposition}\label{pro:heat-kernel-bound}
Suppose $G$ is a connected, locally finite graph satisfying $CDE'(n,K)$ with $K>0$. Then for any $x,y\in V$, $t>0$,
\[p(t,x,y)\leq\frac{1}{\left(1-e^{-\frac{2K}{3}t}\right)^n}.\]
\end{proposition}
\begin{proof}
We apply \eqref{equ:lowb} with $\theta=2K/3$. Considering  $p(\tau,x,y)$, we obtain
\[\partial_\tau \log p(\tau,x,y)\geq -\frac{2nK}{3}\frac{e^{-\theta \tau}}{1-e^{-\theta \tau}}.\]
By integrating from $t$ to $\infty$, and as $\lim_{t\rightarrow\infty} p(t,x,y)=1$, we have
\[p(t,x,y)\leq\frac{1}{\left(1-e^{-\theta t}\right)^n}.\]
This ends the proof.
\end{proof}
\subsection{Diameter bounds}
In this subsection we derive an explicit diameter bound for graphs satisfying $CDE'(n,K)$.

The idea is to prove the operator $\Delta$ satisfies an entropy-energy inequality, as mentioned in the introduction. First we derive, for graphs, an analogue of Davies' theorem(\cite{DB89}) on manifolds.  Note that, obviously, if $\mu$ is a finite measure, $f\in \ell^{\infty}(V,\mu)$ implies $f\in \ell^p(V,\mu)$ for any $p>1$.

\begin{lemma}\label{lem:lsi}  Suppose $G$ is a locally finite, connected graph with $\mu(V)$ bounded.
Let  $f\in \ell^{\infty}(V,\mu)$, satisfying $\|P_tf\|_\infty\leq e^{M(t)}\|f\|_2$, for some continuous and decreasing function $M(t)$.  If  $\|f\|_2=1$, then for any $t > 0$,
\[\sum_{x\in V}\mu(x)f^2(x)\ln f^2(x)\leq 2t\sum_{x\in V}\mu(x)\Gamma(f)(x)+2M(t).\]
\end{lemma}
\begin{proof}
Let $p(s)$ be a bounded, continuous function so that $p(s) \geq 1$ and $p'(s)$ bounded.  For any function $0\leq f\in \ell^{\infty}(V,\mu)$, consider the function $(P_sf)^{p(s)}$. Note $(P_sf)^{p(s)}\in \ell^1(V,\mu)$.  Likewise, so are the functions  $(P_sf)^{p(s)}\ln P_sf$ and $\Delta P_sf(P_sf)^{p(s)-1}\in \ell^1(V,\mu)$.  (Note here, that at $s = 0$, if $f = 0$, we take $(P_s f)^{p(s)}\ln P_s f$ to be zero as well.) so we have
\[\begin{split}
\frac{d}{ds}\|P_sf\|_{p(s)}^{p(s)}
&=\frac{d}{ds}\sum_{x\in V}\mu(x)(P_sf(x))^{p(s)}\\
&=\sum_{x\in V}\mu(x)\frac{d}{ds}(P_sf(x))^{p(s)}\\
&=\sum_{x\in V}\mu(x)\left(p'(s)(P_sf(x))^{p(s)}\ln P_sf(x)+p(s)(P_sf(x))'(P_sf(x))^{p(s)-1}\right)\\
&=p'(s)\sum_{x\in V}\mu(x)(P_sf(x))^{p(s)}\ln P_sf(x)+p(s)\sum_{x\in V}\mu(x)\Delta P_sf(x)(P_sf(x))^{p(s)-1}.
\end{split}\]
At $s=0$, and specializing to $p(s)=\frac{2t}{t-s}$ (where $0\leq s \leq t-t_1$, with $t>t_1>0$), we have
\[\frac{d}{ds}\|P_sf\|_{p(s)}^{p(s)}\mid_{s=0}=\frac{2}{t}\sum_{x\in V}\mu(x)f^2(x)\ln f(x)+2\sum_{x\in V}\mu(x)f(x)\Delta f(x).\]

On the other hand, we give a lower bound on this derivative.
Combining our assumption that $\|P_tf\|_\infty\leq e^{M(t)}\|f\|_2$, for continuous and decreasing $M(t)$ and our assumption that $\|f\|_2 = 1$, and using the Stein interpolation theorem, we obtain
\[\|P_sf\|_{p(s)}\leq e^{\frac{M(t)s}{t}}.\]

Then we obtain
\[\frac{d}{ds}\|P_sf\|_{p(s)}^{p(s)}\mid_{s=0}\leq\frac{2M(t)}{t}.\]
We achieve this using the fact that $\|P_sf\|_{p(s)}^{p(s)}\mid_{s=0} = \|p(s)\|_2 =1$, and $e^{\frac{M(t)sp(s)}{t}}\mid_{s=0}=1$.  Directly computing yields
\[1\geq\lim_{s\rightarrow0^+}\frac{\|P_sf\|_{p(s)}^{p(s)}-1}{e^{\frac{M(t)sp(s)}{t}}-1}=\frac{d}{ds}\|P_sf\|_{p(s)}^{p(s)}\mid_{s=0}\frac{t}{2M(t)}.\]

Noting the identity $-\sum_{x\in V}\mu(x)f(x)\Delta f(x)=\sum_{x\in V}\mu(x)\Gamma(f)(x)$ holds for any $f\in\ell^\infty(V,\mu)$, and combining with the above equality, we obtain
\[\sum_{x\in V}\mu(x)f^2(x)\ln f^2(x)\leq 2t\sum_{x\in V}\mu(x)\Gamma(f)(x)+2M(t),~t>t_1.\]
This completes the proof.
\end{proof}

\begin{proposition}\label{pro:tight-lsi}
Let $G=(V,E)$ be a locally finite, connected graph satisfying $CDE'(n,K)$.  Any $0\leq f\in \ell^{\infty}(V,\mu)$ such that $\|f\|_2=1$ satisfies
\[\sum_{x\in V}\mu(x)f^2(x)\ln f^2(x)\leq \Phi\left(\sum_{x\in V}\mu(x)\Gamma(f)(x)\right),\]
where
\[
\Phi(x)=2n\left[\left(1+\frac{1}{\theta n}x\right)\ln\left(1+\frac{1}{\theta n}x\right)-\frac{1}{\theta n}x\ln\left(\frac{1}{\theta n}x\right)\right].
\]

\end{proposition}
\begin{proof}
Fix such an $f$.  Using Proposition \ref{pro:heat-kernel-bound} and the Cauchy-Schwartz inequality, we have
\[\|P_tf\|_\infty\leq\frac{1}{\left(1-e^{-\theta t}\right)^n}\|f\|_2,\]
where $\theta=\frac{2K}{3}$.
Therefore from Lemma~\ref{lem:lsi}, we obtain
\[\sum_{x\in V}\mu(x)f^2(x)\ln f^2(x)\leq 2t\sum_{x\in V}\mu(x)\Gamma(f)(x)-2n\ln(1-e^{-\theta t}),~t>t_1>0.\]
By minimizing the right-hand side of the above inequality over $t$, we obtain
\[\begin{split}
\sum_{y\in V}\mu(y)f^2(y)\ln f^2(y)
&\leq -\frac{2}{\theta}x\ln\left(\frac{x}{x+\theta n}\right)+2n\ln\left(\frac{x+\theta n}{\theta n}\right)\\
&=2n\left[\left(1+\frac{1}{\theta n}x\right)\ln\left(1+\frac{1}{\theta n}x\right)-\frac{1}{\theta n}x\ln\left(\frac{1}{\theta n}x\right)\right],
\end{split}\]
where $x=\sum_{y\in V}\mu(y)\Gamma(f)(y)$.
\end{proof}
We observe that $\Phi$ is a non-negative, monotonically increasing, and concave function, as we shall use these properties later.

In order to finish the result, and bound the diameter we first need to introduce some notation.  For a positive bounded real valued function $f$ on $V$, let $E(f)$ denote the entropy of $f$ with respect to $\mu$ defined by
\[E(f)=\sum_{x\in V}\mu(x)f(x)\ln f(x)-\sum_{x\in V}\mu(x)f(x)\ln\left(\sum_{x\in V}\mu(x)f(x)\right).\]

To ease the notation, we use $\int f d\mu = \sum_{x\in V}\mu(x)f(x)$. The Laplace operator $\Delta$ satisfies a \textit{logarithmic Sobolev inequality} if there exists a $\rho>0$ such that for all functions $f\in \ell^{\infty}(V,\mu)$,
\[\rho E(f^2)\leq 2\int \Gamma(f) d\mu,\]

Equivalently, it suffices to say that a logarithmic Sobolev inequality holds if all $f\in \ell^{\infty}(V,\mu)$ with $\|f\|_2=1$ satisfy
\begin{equation}\label{equ:lsi}
E(f^2)\leq \Phi\left( \int \Gamma(f) d\mu \right)
\end{equation}
where $\Phi$ is a concave and non-negative function on $[0,\infty)$.

\begin{proposition}\label{pro:intr-diam-bound-lsi}
Suppose $\Delta$ satisfies a general logarithmic Sobolev inequality, and the function $\Phi$ from \eqref{equ:lsi} is non-negative and monotonically increasing.  Then $G$ has diameter
\[\widetilde{D}\leq\sqrt{2}\int_0^\infty\frac{1}{x^2}\Phi(x^2)dx.\]
\end{proposition}
\begin{proof}
Consider any $g\in \ell^{\infty}(V,\mu)$, with $\|\Gamma(g)\|_\infty\leq 1$.
Let $f_\lambda = e^{\lambda g}$ for some $\lambda \in \mathbb{R}^{+}$.  We will apply \eqref{equ:lsi} to the family of non-negative functions $\widetilde{f_\lambda}=\frac{f_{\lambda/2}}{\|f_{\lambda/2}\|_2}$.   Let $G(\lambda)=\|f_{\lambda/2}\|_2^2 = \int e^{\lambda g} d\mu$ and observe that $G'(\lambda)=\int ge^{\lambda g}d\mu ~ \left(=\frac{1}{\lambda}\int f_{\lambda/2}^2\ln f_{\lambda/2}^2d\mu\right)$.


On one hand, it is immediate by the definition that, $\widetilde{f}$,
\[E(\widetilde{f}_\lambda^2)=\frac{1}{G(\lambda)}\left(\lambda G'(\lambda)-G(\lambda)\ln G(\lambda) \right).  \]

We also must consider the right hand side of the Sobolev inequality, which contains a term of the form $\int \Gamma(\widetilde{f}_\lambda) d\mu = \frac{1}{\|f_{\lambda/2}\|_2^2}\int \Gamma(e^{\frac{\lambda g}{2}}) d\mu$.   Such terms can be bounded as follows:
\begin{align*}
\int \Gamma(e^{\frac{\lambda g}{2}}) d\mu
&=\frac{1}{2}\sum_{x\in V} \mu(x) \sum_{y\sim x}\omega_{xy}\left(e^{\frac{\lambda g(y)}{2}}-e^{\frac{\lambda g(x)}{2}}\right)^2\\
&=\frac{1}{2}\sum_{x\in V} \mu(x) \sum_{\tiny\begin{array}{c} y\sim x\\g(x)>g(y)\end{array}}\omega_{xy}\left(e^{\frac{\lambda g(y)}{2}}-e^{\frac{\lambda g(x)}{2}}\right)^2 \\&\hspace{1in} +\frac{1}{2}\sum_{x\in V}\mu(x)\sum_{\tiny\begin{array}{c} y\sim x\\g(x)<g(y)\end{array}}\omega_{xy}\left(e^{\frac{\lambda g(y)}{2}}-e^{\frac{\lambda g(x)}{2}}\right)^2\\
&=\sum_{x\in V} \mu(x) \sum_{\tiny\begin{array}{c} y\sim x\\g(x)>g(y)\end{array}}\omega_{xy}\left(e^{\frac{\lambda g(y)}{2}}-e^{\frac{\lambda g(x)}{2}}\right)^2\\
&\leq\sum_{x\in V} \mu(x) \sum_{\tiny\begin{array}{c} y\sim x\\g(x)>g(y)\end{array}}\omega_{xy}\left(e^{\frac{\lambda}{2}(g(y)-g(x))}-1\right)^2 e^{\lambda g(x)}\\
&\leq\frac{\lambda^2}{4}\sum_{x\in V} \mu(x)  e^{\lambda g(x)}\sum_{\tiny\begin{array}{c} y\sim x\\g(x)>g(y)\end{array}}\omega_{xy}(g(y)-g(x))^2\\
&\leq\frac{\lambda^2}{2}\int e^{\lambda g}\Gamma(g) d\mu.
\end{align*}
Since $\Gamma(g) \leq 1$, and the function $\Phi$ is monotonically increasing
\[\Phi\left( \int \Gamma (\widetilde{f_\lambda}) d\mu\right)= \Phi\left(\frac{1}{\|f_{\lambda/2}\|_2^2}\int \Gamma(e^{\frac{\lambda g}{2}}) d\mu \right)\leq \Phi\left(\frac{\lambda^2}{2}\right).\]
By the logarithmic Sobolev inequality
\[\lambda G'(\lambda)-G(\lambda)\ln G(\lambda)\leq G(\lambda)\Phi\left(\frac{\lambda^2}{2}\right).\]
Let $H(\lambda)=\frac{1}{\lambda}\ln G(\lambda)$. Then the above inequality reads
\[H'(\lambda)\leq \frac{1}{\lambda^2}\Phi\left(\frac{\lambda^2}{2}\right).\]
Since $H(0)=\lim_{\lambda\rightarrow 0}\frac{1}{\lambda}\ln G(\lambda)= \int g d\mu$, it follows that
\[H(\lambda)=H(0)+\int_0^\lambda H'(u)du\leq \int g d\mu + \int_0^\lambda\frac{1}{u^2}\Phi\left(\frac{u^2}{2}\right)du.\]
Therefore for any  $\lambda \geq 0$,
\begin{equation}
\sum_{x\in V}\mu(x)e^{\lambda(g(x)-\int g d\mu)}\leq \exp\left\{\lambda \int_0^\lambda\frac{1}{u^2}\Phi\left(\frac{u^2}{2}\right)du\right\}.
\end{equation}
Let $C=\int_0^\infty\frac{1}{u^2}\Phi\left(\frac{u^2}{2}\right)du=\frac{1}{\sqrt{2}}\int_0^\infty\frac{1}{x^2}\Phi(x^2)dx$.  Then, for every $\lambda \geq 0$ and $\epsilon > 0$ when we apply the above inequality to $g$ and $-g$ and apply Chebyshev's inequality,
\begin{align*}
&\mu(\{x\in V:|g(x)-\int g d\mu|\geq C+\varepsilon\})\\
&\leq \sum_{\substack{x \in V\\
g(x) \geq \int gd\mu + C+\varepsilon}}\mu(x)+ \sum_{\substack{x \in V \\g(x)\leq \int g d\mu - C-\varepsilon}}\mu(x)\\
&\leq \sum_{\substack{x \in V\\ g(x) \geq \int g d\mu +  C+\varepsilon}}\frac{e^{\lambda (g(x)-\int g d\mu)}}{e^{\lambda(C+\varepsilon)}}\mu(x)+ \sum_{\substack{x \in V \\g(x) \leq  \int gd\mu  -C-\varepsilon}}\frac{e^{\lambda (-g(x)+\int g d\mu)}}{e^{\lambda(C+\varepsilon)}}\mu(x)\\
&\leq 2e^{-\lambda(C+\varepsilon)}e^{\lambda C}\\
&=2e^{-\lambda \varepsilon}\rightarrow 0(\lambda\rightarrow \infty).
\end{align*}
\noindent That is, we obtain
\[\|g(x)-\int g d\mu\|_\infty\leq C.\]
The diameter bounds follows immediately by the definition of $\widetilde{D}$:  Since $g$ was arbitrary,
\[\widetilde{D}\leq\sqrt{2}\int_0^\infty\frac{1}{x^2}\Phi(x^2)dx.\]
That completes the proof.
\end{proof}

Finally, we obtain
\begin{theorem}\label{th:intr-diam-bound}
Let $G=(V,E)$ be a locally finite, connected graph satisfying $CDE'(n,K)$, and $K>0$, then the diameter $\widetilde{D}$ of graph $G$ in terms of canonical distance is finite, and
\[\widetilde{D}\leq 4\sqrt{3}\pi \sqrt{\frac{n}{K}}.\]
\end{theorem}
\begin{proof}
Combining Proposition~\ref{pro:tight-lsi} and Proposition~\ref{pro:intr-diam-bound-lsi}, graphs satisfying $CDE'(n,K)$ for some $K>0$, also satisfy
\[\widetilde{D}\leq\sqrt{2}\int_0^\infty\frac{1}{x^2}\Phi(x^2)dx,\]
where $\Phi(x)=2n\left[\left(1+\frac{1}{\theta n}x\right)\ln\left(1+\frac{1}{\theta n}x\right)-\frac{1}{\theta n}x\ln\left(\frac{1}{\theta n}x\right)\right]$, and $\theta=\frac{2K}{3}$.
Since
\[\int_0^\infty\frac{1}{x^2}\Phi(x^2)dx
=\frac{1}{2}\int_0^\infty\frac{1}{x^{\frac{3}{2}}}\Phi(x)dx
=\int_0^\infty\frac{1}{\sqrt{x}}\Phi'(x)dx
=-2\int_0^\infty\sqrt{x}\Phi''(x)dx<\infty
\]
then the diameter is finite, and $\Phi''(x)=-\frac{2n}{x(x+\theta n)}$, then
\[-2\int_0^\infty\sqrt{x}\Phi''(x)dx=4\pi\sqrt{\frac{n}{\theta}},\]
so we have completed the proof.
\end{proof}

While this bounds the canonical diameter, it is possible to recover a bound for the usual graph distance. In order to accomplish this, we first introduce the notion of an intrinsic metric.  This gives us a way to relate  natural distance with the canonical distance.

Intrinsic metrics on graphs was first introduced by R. Frank, D. Lenz and D. Wingert in \cite{FLW14}.  A function $\rho: V\times V\rightarrow \R^+$ is called an intrinsic metric if, at all $x \in V$
\[\sum_{y\sim x}\omega_{xy}\rho^2(x,y)\leq \mu(x).  \]
This induces a metric $\rho$ on a graph via finding shortest paths.
One example of such a function, introduced by Xueping Huang in his thesis (\cite{H11}), is to define for all for all $x\in V$ and $y\sim x$,
\[\widetilde{\rho}(x,y)=\min\left\{\sqrt{\frac{\mu(x)}{m(x)}},\sqrt{\frac{\mu(y)}{m(y)}}\right\},\]
where $m(x)=\sum_{y\sim x}\omega_{xy}$.

As mentioned, these metrics give a way of comparing distances with the intrinsic distance we have been using.  Indeed, part $(a)$ of the remark following Definition 1.2 in \cite{KLSW15} gives:
\begin{proposition}\label{pro:intr-diameter}
For any $x,y\in V$, it holds that
\[\sqrt{2}\widetilde{\rho}(x,y)\leq \widetilde{d}(x,y).\]
\end{proposition}
Actually from \cite{KLSW15}, the above inequality is true for any intrinsic metric.  For the metric $\tilde{\rho}$ however, under the assumption $D_\mu$ is finite, then
$$ \widetilde{\rho}(x,y)\geq \frac{d(x,y)}{\sqrt{D_\mu}}$$
for any $x$ and $y$ (by, again, extending the metric $\tilde{\rho}$ along shortest paths.)

The above inequality and Theorem~\ref{th:intr-diam-bound} and Proposition~\ref{pro:intr-diameter} combine to prove the following inequality
\begin{theorem}\label{diameterbound}
If a graph be a locally finite, connected, and satisfy $CDE'(n,K)$ with $K>0$, then the diameter of the natural distance on the graph is finite, moreover the upper bound quantitative estimation is
\[D\leq 2\pi\sqrt{\frac{6D_\mu n}{K}}.\]
\end{theorem}

\bibliographystyle{amsalpha}

\begin{thebibliography}{A}

\bibitem[B94]{B94} D. Bakry, \textit{L'hypercontractivit\'{e} et son utilisation en th\'{e}orie des semigroupes}, Lectures on Probability Theory, Saint-Flour, 1992, 1-114. Lecture Notes in Math., vol. 1581. Springer, Berlin,1994.
\bibitem[BE83]{BE83} D. Bakry, M. Emery, \textit{Diffusions hypercontractives}. (French) [Hypercontractive diffusions]
Seminaire de probabilite, XIX, 1983/84, 177--206, Lecture Notes in Math., \textbf{1123}, Springer,
Berlin, 1985.
\bibitem[Bi63]{Bi63} R. Bishop, {\it A relation between volume, mean curvature , and diameter}, Amer. Math. Soc. Not.(1963), 10 p. 364.
\bibitem[Bu82]{Bu82} P. Buser,{\it A note on the isoperimetric constant}, Ann. Sci. Ecole Norm. Sup. (4) \textbf{15}, no. 2,  (1982) 213-230.
\bibitem[BL96]{BL96}D. Bakry, M. Ledoux, \textit{Sobolev inequality and Myer's diameter theorem for an abstract Markov generator}, Duke Mathmatical Journal,(1996), Vol.85, No.1.
\bibitem[BL06]{BL06}D. Bakry, M. Ledoux, \textit{A logarithmic Sobolev form of the Li-Yau gradient estimate}, Mat. Iberoamericana 22, no. 2 (2006), 683-702.
\bibitem[BQ99]{BQ99}D. Bakry, Z.M. Qian, \textit{Harnack inequalities on a manifold with positive or negative Ricci curvature}, Rev. Mat. Iberoamericana 15 (1999), no.1, 143-179.
\bibitem[BG$09$]{BG$09$} F. Baudoin, N. Garofalo, \textit{Perelman's entropy and doubling property on Riemanmmian manifolds}, Journal of Geometric Analysis, (2009), 1119-1131.
\bibitem[BG$11$]{BG$11$}F. Baudoin, N. Garofalo, \textit{Curvature-dimension inequalities and Ricci lower bounds for sub-Riemannian manifolds with transverse symmetries}, arXiv preprint, (2011).
\bibitem[BHLLMY13]{BHLLMY13} F. Bauer, P. Horn, Y. Lin, G. Lippner, D. Mangoubi and S.-T. Yau, \textit{Li-Yau inequality on graphs}, Arxiv:1306.2561v2, (2013), to appear in Journal of Differential Geometry.
\bibitem[CY96]{CY96} F. Chung, S-T Yau, \textit{Logarithmic Harnack inequalities}, Math. Res. Lett, 3 (1996), 793-812.
\bibitem[CM96]{CM96} T. H. Colding, W. P. Minicozzi II, \textit{Generalized Liouville properties of manifolds}, Math. Res. Lett, 3 (1996), 723-729.
\bibitem[CG98]{CG98} T. Coulhon, A. Grigoryan, \textit{Random walks on graphs with regular volume growth}, Geometric and Functional Analysis, Springer, (1998).
\bibitem[D97]{D97} T. Delmotte,  \textit{Harnack Inequalities on graphs},  S$\acute{e}$minaire de Th$\acute{e}$orie spectrale et g$\acute{e}$om$\acute{e}$tric, tome 16(1997-1998), 217-228.
\bibitem[D99]{D99} T. Delmotte,  \textit{Parabolic Harnack Inequality and estimates of Markov chains on graphs}, Geometric and Functional Analysis, Revista Matem atica Iberoamericana, 15 (1999), 181-232.
\bibitem[DB$89$]{DB89}Davies, E.B. \textit{Heat kernels and spectral theory},  Cambridge Tracts in Mathematics, 92. Cambridge University Press, Cambridge, (1989).
\bibitem[DB$93$]{DB93}Davies, E.B. \textit{Large derivations for Heat kernels on graphs}, Journal of London Mathematical Society, 47(1993), 65-72.
\bibitem[EKS$13$]{EKS13}M. Erbar, K. Kuwada, K.T. Sturm, \textit{On the equivalence of the entropic curvature-dimension condition and Bochner's inequality on metric measure spaces}, arXiv:1303.4382, to appear in Inventions Math.
\bibitem[FLW14]{FLW14} R. Frank, D.Lenz, D. Wingert,  \textit{Intrinsic metrics for non-local symmetric Dirichlet forms and applications to spectral
theory}, J. Funct. Anal. 266 (2014), no. 8, 4765?-4808.
\bibitem[G92]{G92} A. Grigor'yan, \textit{The heat equation on noncompact Riemannian manifolds}, Math. USSR Sb. 72(1992), 47--77.
\bibitem[GH14]{GH14}  A. Grigor¡¯yan, J. Hu, \textit{Upper bounds of heat kernels on doubling spaces}, Moscow Math. J. 14 (2014) 505-563.
\bibitem[HKLW12]{HKLW12} S. Haeseler, M. Keller, D. Lenz, R. Wojciechowski, \textit{Laplacians on infinite graphs: Dirichlet and Neumann
boundary conditions}, J. Spectr. Theory, 2 (2012), no. 4, 397?432.
\bibitem[HS93]{HS93} W. Hebisch, L. Saloffe-Coste, \textit{Gaussian estimates for Markov chains and Random walks on groups}, Ann. Probab., 21(2)1993,
673-709.
\bibitem[H11]{H11} X.P. Huang,  \textit{On stochastic completeness of weighted graphs}, PH.D thesis, University of Bielefeld, 2011.
\bibitem[JLZ14]{JLZ14}Renjin Jiang, Huaiqian Li, Huichun Zhang, \textit{Heat Kernel Bounds on Metric Measure Spaces and Some Applications}, preprint.
\bibitem[L95]{L95}M.Ledoux, \textit{Remarks on logarithmic Sobolev constants, exponential integrability and bounds on the diameter}, Journal of Mathematics of Kyoto University, (1995).
\bibitem[LP14]{LP14} S.P. Liu, N. Peyerimhoff, \textit{Eigenvalue ratios of nonnegatively curved graphs},  Arxiv:1406.6617v1, 2014.
\bibitem[Li97]{Li97}P. Li,  \textit{Harmonic sections of polynomial growth}, Math. Res. Lett, 4 (1997), 35-44.
\bibitem[LY86]{LY86}P. Li, S-T Yau, \textit{On the parabolic kernel of the Schr\"{o}dinger operator}, Acta Math. 156 (1986), 153-201.
\bibitem[LY10]{LY10}Y. Lin, S.T.Yau, \textit{Ricci curvature and eigenvalue estimate on locally finite graphs}, Math. Res.Lett. (2010),343-356.
\bibitem[KL12]{KL12}M., Keller, D. Lenz,  \textit{Dirichlet forms and stochastic completeness of graphs and subgraphs}, J. Reine Angew. Math. 666 (2012), 189-223.
\bibitem[KLSW15]{KLSW15} M., Keller, D. Lenz, M. Schmidt, M. Wirth \textit{Diffusion determines the recurrent graph}, Adv. Math. 269 (2015), 364-398.
\bibitem[P93]{P93}Pang, M.M.H. \textit{Heat kernels of graphs},  Journal of London Mathematical Society, 47(1993), 50-64.
\bibitem[R12]{R12}T. Rajala, \textit{Interpolated measures with bounded density in metric spaces
satisfying the  curvature-dimension conditions of Sturm}, J. Funct. Anal. 263 (2012), 896-924.
\bibitem[SC95]{SC95} L. Saloff-Coste, \textit{Parabolic Harnack inequality for divergence form second order differential operators}, Potential analysis, 4(1995), 429--467.
\bibitem[SZ97]{SZ97} D. W. Stroock,  W. Zheng,   \textit{Markov chain approximations to symmetric diffusions}, Ann.I.H.P., 33(1997), 619-649.
\bibitem[S06]{S06} K.T. Sturm, \textit{On the geometry of metric measure spaces II}, Acta Math. 196 (2006), 133-177.
\bibitem[WE10]{WE10} A. Weber, \textit{Analysis of the physical Laplacian and the heat flow on a locally finite graph}, Journal of Mathematical Analysis and Applications, (2010), 146-158.
\bibitem[WO09]{WO09} R. Wojciechowski,   \textit{Heat kernel and essential spectrum of
infinite graphs}, Indiana Univ. Math. J. 58 (2009), no. 3, 1419-?1441.
\bibitem[Yau86]{Yau86} S. T. Yau, \textit{Nonlinear analysis on geometry}, L'enseignement Math$\acute{e}$matique, SRO-KUNDIG, Gen$\grave{e}$ve, 1986.
\end{thebibliography}

Paul Horn, \\
Department of Mathematics, Denver University, Denver, Colorado, USA \\
\textsf{Paul.Horn@du.edu}\\
Yong Lin,\\
Department of Mathematics, Renmin University of China, Beijing, China\\
\textsf{linyong01@ruc.edu.cn}\\
Shuang Liu,\\
Department of Mathematics, Renmin University of China, Beijing, China\\
\textsf{cherrybu@ruc.edu.cn}\\
Shing-Tung Yau,\\
Department of Mathematics, Harvard University, Cambridge, Massachusetts, USA\\
\textsf{yau@math.harvard.edu}\\
\end{document}